\documentclass[11 pt,leqno]{amsart}

\usepackage[all]{xy}
\usepackage {amsfonts}
\usepackage{amsthm}
\usepackage{amssymb}
\usepackage{latexsym}
\usepackage{amsmath}
\usepackage{mathrsfs}
\usepackage{dsfont}
\usepackage{a4wide}
\usepackage{hyperref}

\usepackage{enumitem} 
\setlist[enumerate,1]{label=\textup{(\arabic*)}}

\theoremstyle{plain}
\newtheorem{thm}{Theorem}[section]

\newtheorem {lem} [thm]{Lemma}
\newtheorem {prop}[thm] {Proposition}

\newtheorem{cor}[thm]{Corollary}

\theoremstyle{definition}
\newtheorem {defn}[thm] {Definition}
\newtheorem {rem} [thm]{Remark}
\newtheorem{ex}[thm]{Example}

\newcommand{\eqn}{\begin{equation}}
\newcommand{\eqne}{\end{equation}}

\DeclareMathOperator{\Res}{Res}
\DeclareMathOperator{\Ex}{Ex}
\DeclareMathOperator{\Prime}{Prime}
\DeclareMathOperator{\tr}{tr}

\DeclareMathOperator{\Aut}{Aut}

\DeclareMathOperator{\Iso}{Iso}
\newcommand{\BB}{B}
\DeclareMathOperator{\Homeo}{Homeo}
\DeclareMathOperator{\Mult}{\mathcal{M}}

\DeclareMathOperator{\red}{r}

\DeclareMathOperator{\clsp}{\overline{span}}

\DeclareMathOperator{\Null}{\mathcal{N}}

\DeclareMathOperator{\supp}{supp}

\newcommand{\I}{\mathcal I}
\DeclareMathOperator{\Ind}{Ind}
\newcommand{\B}{\mathcal B}

\newcommand{\M}{\mathcal M}

\newcommand{\EE}{\mathcal E}
\newcommand{\E}{\mathbb E}
\newcommand{\RR}{\mathcal R}

\newcommand{\OO}{\mathcal{O}}

\newcommand{\F}{\mathbb F}

\newcommand{\FF}{\mathcal F}

\newcommand{\C}{\mathbb C}

\newcommand{\Z}{\mathbb Z}
\newcommand{\N}{\mathbb N}

\newcommand{\cst}{\ifmmode\mathrm{C}^*\else{$\mathrm{C}^*$}\fi}

\newcommand*{\onto}{\twoheadrightarrow}

\begin{document}

\author{Krzysztof Bardadyn}
	\address{Faculty of Mathematics, University of Bia\l ystok, ul. K. Cio\l kowskiego 1M, 15-245
	Bia\l ystok, Poland}
	\email{kbardadyn@math.uwb.edu.pl}

\author{Bartosz Kwa\'sniewski}
	\address{Faculty of Mathematics, University of Bia\l ystok, ul. K. Cio\l kowskiego 1M, 15-245
	Bia\l ystok, Poland}
	\email{b.kwasniewski@uwb.edu.pl}


\title{\small Topologically free actions and ideals in 
\\
 twisted  Banach algebra crossed products}

\begin{abstract} 
 We generalise the influential $C^*$-algebraic results of Kawamura-Tomiyama and Archbold-Spielberg for crossed products of discrete groups actions to the realm of Banach algebras and twisted actions. 
We prove that topological freeness is equivalent to the intersection property for all reduced twisted Banach algebra crossed products coming from subgroups, and in the untwisted case to a generalised intersection property for a full $L^p$-operator algebra crossed product for  any $p\in [1,\infty]$.  This gives efficient simplicity criteria for various Banach algebra crossed products. We also use it to identify  the prime ideal space of some crossed products as the  quasi-orbit space of the action. 
For amenable actions we prove that the full and reduced twisted $L^p$-operator algebras coincide.  
\end{abstract}
\maketitle

\section{Introduction}

Topological freeness is a key condition in the study of the   $C^*$-algebras associated to discrete group actions  and their generalisations, see 
\cite{Zeller-Meier}, \cite{OD}, \cite{OP}, \cite{Kawa-Tomi}, \cite{Tomiyama}, \cite{Arch_Spiel}, \cite{Exel_Laca_Quigg}, \cite{Kwa-Meyer-1}, \cite{Kwa-Meyer}, \cite{Kwa-Meyer2}. It is crucial for the description of the ideal structure and for pure infiniteness criteria of various $C^*$-algebras, and is known to be closely related to such conditions  as  pure outerness, proper outerness, full Connes spectrum, aperiodicity,   almost extension property, uniqueness of a pseudo-expectation,  detection of ideals, supporting, filling, maximal abelian  or Cartan inclusions, see     \cite{Kwa-Meyer-1}, \cite{Kwa-Meyer2} and references therein. Topological freeness is also important in the spectral analysis of weighted composition operators on $L^2$-spaces, see \cite{OD},  \cite{Anton_Lebed}, \cite{bar_kwa}.
The title of the present article alludes to the seminal paper  \cite{Arch_Spiel}, as we generalise the corresponding results to the setting of twisted Banach algebra crossed products.

In the 2010s, Phillips initiated a program of generalising important $C^*$-algebraic constructions and results  to Banach algebras represented on $L^p$-spaces for $p\in[1,\infty)$. In recent years the number of papers in this direction is rapidly growing,
 see the survey paper \cite{Gardella} and/or \cite{Phillips}, \cite{PH},  \cite{Gardella_Lupini17}, \cite{Chung}, \cite{Gardella},  \cite{cgt},\cite{Austad_Ortega},  \cite{Wang_Zhu}, \cite{Hetland_Ortega}, \cite{BKM1}, and the increasing importance of this field of research is   reflected in the growing number of publications in top notch journals, see
\cite{Gardella_Thiel16}, \cite{Gardella_Lupini17}, \cite{Gardella_Thiel19}, \cite{Boedihardjo}, \cite{Gardella_Thiel22}, \cite{Hetland_Ortega}.
In particular, Banach algebra crossed products   \cite{DDW}, and so also $L^p$-operator algebra crossed products are one of the most fundamental constructions, cf. \cite{Phillips}, \cite{PH}, \cite{Chung}, \cite{Gardella},  \cite{cgt}, \cite{Wang_Zhu}. 

 In this article, we contribute  to the development of this modern theory by solving  one of the fundamental problems which is a Banach algebra version of the well known $C^*$-algebraic simplicity criterion.
In particular, Phillips \cite[Question 8.2]{Phillips} and   Gardella-Lupini \cite[Problem 8.2]{Gardella_Lupini17}  ask  whether topological freeness and minimality of a discrete  group action $\varphi: G\to \Homeo(X)$, on a locally compact Hausdorff space $X$, are sufficient for simplicity of the associated reduced $L^p$-operator algebra crossed product. 
For $C^*$-algebras this is a well-known fact that follows from a theorem of Kawamura-Tomiyama \cite{Kawa-Tomi},  generalised by  Archbold-Spielberg  in \cite{Arch_Spiel}, which characterises topological freeness of a discrete group action  in  terms of a generalised intersection property for the full crossed product. 
This implies simplicity of reduced crossed products for minimal actions, and shows that topological freeness is  necessary for simplicity of the full crossed product.

In this  paper we generalise the results of Kawamura-Tomiyama and Archbold-Spielberg in two directions. 
Firstly, we 
extend it 
not only to $L^p$-operator algebras, $p\in [1,\infty]$, but to a much more general class of Banach algebra crossed products that we introduce in the paper (see Theorem \ref{thm:generalised_intersection_property}, Corollary \ref{cor:general_generalised_intersection}).
Secondly, we prove a version of it for twisted actions (Theorem \ref{thm:twisted_generalised_intersection_property}), which is inspired by recent $C^*$-algebraic results \cite{Kwa-Meyer2}, 
characterising the generalised intersection property for every intermediate $C^*$-algebra.
Our proofs differ from the original $C^*$-algebraic arguments that usually exploit positivity, a concept absent in the   Banach algebra setting.
One of the main  ideas to overcome this is to use not only the canonical expectation but the whole family of `expectations', indexed by elements of the group $G$, 
that we refer to as the \emph{Fourier decomposition}. 
An important role is also played by the minimality of the supremum norm in the algebra $C_0(X)$, that $G$ acts upon.
Another new ingredient is that in order to have an `if and only if' result for twisted actions, we  consider  crossed products coming from subgroups $H\subseteq	G$.
Here in the proof we had to develop new techniques. In particular we use that any second cohomology group   for free groups is trivial and that crossed products for a trivial action can be described as section algebras of a bundle of twisted group algebras. One of the main applications  (see Theorem \ref{cor:simplicity_minimality}) is that for a topologically free  action $\alpha$ on  the algebra $C_0(X)$, and any twist $u\in H^2(G, C_u(X))$, 
a reduced twisted  Banach algebra crossed product $F_{\RR}(\alpha,u)$ is \emph{simple} if and only if the action is minimal.

A notion of  a \emph{reduced Banach algebra crossed product} that we introduce (in Defintion \ref{def:reduced_crossed_product}) 
generalises reduced group Banach algebras defined recently by Phillips in \cite{Phillips19},  
and has reduced twisted $L^p$-operator algebra crossed products as examples. 
In fact, all of our results apply to a more general class of  $L^P$-operator algebra crossed products
where $P\subseteq [1,\infty]$ is a set of H\"older exponents (Definition \ref{def:L^P_crossed_products}). 
This class includes as special cases different algebras studied by other authors  and allows for nice phrasing of our results. We prove that for \emph{amenable actions} the full and reduced $L^P$-crossed products coincide (Theorem \ref{thm:ambenability_reduced_equals_full} and Remark \ref{rem:properties_of_F^P}),
and thus for amenable actions  simplicity of any $L^P$-crossed product is equivalent to topological freeness and minimiality of the action.

We also introduce \emph{residually reduced Banach algebra crossed products} (Definition \ref{def:residually_reduced}), which 
for $C^*$-algebras coincide with  reduced crossed products for exact actions  \cite{Sierakowski}. 
The $L^P$-operator algebra crossed products for $P\subseteq \{1,\infty\}$ are always residually reduced. We show that if $F_{\RR}(\alpha,u)$ is residually reduced and the action is residually topologically free, then   the \emph{ideal lattice} of   $F_{\RR}(\alpha,u)$ is canonically isomorphic to the lattice of $G$-invariant open subsets of $X$ (Theorem \ref{thm:ideal structure}). If in addition $X$ is second countable, this isomorphism induces a homeomorphism from the prime ideal space $\Prime(F_{\RR}(\alpha,u))$ onto the quasi-orbit space $\OO(X)$ of the action (Corollary \ref{cor:effros-hahn}). 
This generalises, and in some cases improves, well known results from $C^*$-algebra theory to a much more general Banach algebra framework.
  
The paper is organised as follows.
 In Section \ref{sec:preliminaries}, we fix notation and discuss  basic facts concerning the Banach algebra   $F(\alpha,u)=\ell^{1}(G,C_0(X))$ associated to a twisted action $(\alpha,u)$ on $C_0(X)$. We also introduce full and reduced  twisted $L^p$-operator algebra crossed products $F^p(\alpha,u)$ and $F^{p}_{\red}(\alpha,u)$,  $p\in[1,\infty]$. 
Section 3 is devoted to a detailed proof, built on ideas of \cite{cgt}, that for amenable actions we always have $F^{p}(\alpha,u)=F^{p}_{\red}(\alpha,u)$. In Section \ref{sec:reduced crossed products}, we introduce reduced Banach algebra crossed products and more general Banach algebra crossed products $F_{\RR}(\alpha,u)$ that admits the Fourier decomposition. 
We also  introduce algebras $F^{P}(\alpha,u)$ and $F^{P}_{\red}(\alpha,u)$ for any set of parameters $P\subseteq [1,\infty]$. Section \ref{sec:the_main_result} presents the proofs of our main theorems. Finally, in Section \ref{sec:consequences} we show how they can be used to derive simplicity, or more generally how to describe the ideal structure of  residually reduced crossed products.

\subsection*{Acknowledgements}

The authors were supported by the National Science Center (NCN) grant no.~2019/35/B/ST1/02684.
We would like to thank Andrew McKee for a number of useful comments, especially those on amenable actions, and 
 to an  anonymous reviewer for a number of suggestions that improved the presentation.

\section{Preliminaries}
\label{sec:preliminaries}
Throughout this paper we fix an action $\varphi$ of a discrete group $G$ on a   locally compact Hausdorff space $X$. Thus  $\varphi: G\to \Homeo(X)$ is a group homomorphism into a group of homeomorphisms.
Equivalently, we fix a group action  $\alpha:G\to \Aut(C_0(X))$ on the commutative $C^*$-algebra $C_0(X)$ of continuous functions vanishing at infinity;
as then  we necessarily have $\alpha_s(a)=a\circ\varphi_{s^{-1}}$, $s\in G$, for some action $\varphi$ on $X$. This action uniquely extends to the action $\alpha:G\to \Aut(C_b(X))$ on the $C^*$-algebra  $C_b(X)$ of continuous bounded functions
(the multiplier algebra  of $C_0(X)$), which in turn restricts to an action on the unitary group
 $C_u(X)=\{ a\in C_b(X): |a(x)|=1 \text{ for all }x\in X\}$   of  $C_b(X)$. 
Thus we may view the abelian group $C_u(X)$ as a $G$-module, that yields a  cohomology for $G$. 
We recall that the corresponding group of $2$-\emph{cocycles} $Z^2(G,C_u(X))$  consists of  
maps
$u: G\times G\to C_u(X)$  satisfying
\begin{equation}\label{eq:cocycle_identity}
\alpha_r(u(s,t))u(r,st)=u(r,s)u(rs,t), \qquad r,s,t\in G.
\end{equation}
The \emph{second cohomology group} $H^2(G, C_u(X))$ is the quotient of $Z^2(G,C_u(X))$ by the subgroup of \emph{$2$-cobundaries}, which are  maps of the form $d^2(\xi)(s,t):=\alpha_s(\xi(t))\xi(st)^{-1}\xi(s)$, $s,t\in G$, for some map $\xi:G\to C_u(X)$.
Thus two cocycles $u,v$ are \emph{equivalent} (\emph{cohomologous}) if $u(s,t)\xi(st)=v(s,t)\alpha_s(\xi(t))\xi(s)$ for some 
 $\xi: G\to C_u(X)$ and all $s,t\in G$.
We say that $u$ is \emph{normalised} if $u(1,1)=1$. Identity \eqref{eq:cocycle_identity} 
implies that $u(1,t)=u(1,1)$ and $u(t,1)=\alpha_t(u(1,1))$ 
  for all $t\in G$.  
Hence
if $u$ is normalised, then $u(t,1)=u(1,t)=1$ for all $t\in G$. 
Every cocycle is equivalent to a normalised one becasue $u \cdot d^2(\xi)$ is normalised for any $u$ and $\xi$ with  $\xi(1)=\overline{u(1,1)}$.  

To any cocycle $u\in Z^2(G,C_u(X))$ we associate a Banach $*$-algebra 
 $$
F(\alpha,u):=\ell^{1}(G,C_0(X))=\{a:G\to C_0(X):  \|a\|_1=\sum_{g\in G}\|a(g)\|< \infty\}
$$
 equipped with multiplication and involution  given by 
\begin{equation}\label{eq:crossed_product_operations}
(a * b)(r)=\sum_{st=r}a(s)\alpha_s(b(t))u(s,t), \qquad
a^*(t)=\alpha_t(a(t^{-1})^*)u(t,t^{-1})^*u(1,1)^*
\end{equation}
(associativity of the product $*$ is equivalent to 
the identity \eqref{eq:cocycle_identity}, and it implies that $^*$ is an involution), cf. \cite[2.4]{Zeller-Meier}. 
Viewing $F(\alpha,u)$ as  the closed linear span of elements $a_s\delta_s$ where $\delta_s$ is the indicator of  $s\in G$  and $a_s \in C_0(X)$, we get 
$$
(a_{s}\delta_{s})(a_{t}\delta_{t})=a_{s}\alpha_{s}(a_{t})u(s,t)\delta_{st}, \quad  (a_t\delta_t)^*=\alpha_{t^{-1}}(a_t^*)  u(t^{-1},t)^*u(1,1)^*\delta_{t^{-1}},
$$ 
for all $s,t\in G$, $a_{s},a_{t}\in C_0(X)$. In particular, $F(\alpha,u)=\overline{\bigoplus_{t\in G} C_0(X) \delta_t}$ is a $G$-graded $*$-algebra, that is  $(C_0(X) \delta_s)\cdot(C_0(X) \delta_t)=C_0(X) \delta_{st}$
and $(C_0(X)\delta_t)^*=C_0(X)\delta_{t^{-1}}$. Also $C_0(X)\ni a\mapsto a u(1,1)^*\delta_1\in C_0(X)\delta_1$ is an isometric isomorphism of Banach algebras allowing us to identify $C_0(X)$ with a  subalgebra of $F(\alpha,u)$. 
Any approximate unit in $C_0(X)$ is  an approximate unit in $F(\alpha,u)$.
\begin{defn}
We call $F(\alpha,u)$  the  \emph{universal Banach algebra crossed product} for $(\alpha,u)$. 
When $u\equiv 1$ we write $F(\alpha):=F(\alpha,1)$.
\end{defn}

Up to a natural isomorphism  $F(\alpha,u)$ only depends on the equivalence class of $u$:
\begin{prop}[cf. {\cite[Proposition 2.5]{Zeller-Meier}}]\label{prop:equivalent_cocycles} 
Cocycles $u,v\in Z^2(G,C_u(X))$  are equivalent if and only if $F(\alpha,u)\cong F(\alpha,v)$ by an  isometric isomorphism  that respects the $G$-gradings and it is the identity on $C_0(X)$.
Any such isomorphism $\psi:F(\alpha,u) \to F(\alpha,v)$   is  
 $*$-preserving  and is of the form
\begin{equation}\label{eq:Multiplier-isomorphism}
 \psi(b)(t)=b(t) \xi(t), \qquad b\in F(\alpha,u), t\in G,
\end{equation}
where $\xi:G\to C_u(X)$ is such that $u=vd^{2}(\xi)$. Thus such isomorphisms 
are parametrised 
by $1$-cocycles $Z^1(G,C_u(X)):=\{\omega:G\to C_u(X): \omega(st)=\omega(s)\alpha_s(\omega(t))\text{ for }s,t \in G\}$.
\end{prop}
\begin{proof}
Clearly, for any function  $\xi: G\to C_u(X)$  the formula \eqref{eq:Multiplier-isomorphism} defines a surjective $C_0(X)$-module isometry $\psi: F(\alpha,u)\to F(\alpha,v)$ that respects the $G$-grading. 
Since 
$$
\psi(a_s\delta_s \cdot a_t\delta_t)=a_{s}\alpha_{s}(a_{t})u(s,t)\xi(st)\delta_{st},\quad
\psi(a_s\delta_s) \cdot \psi(a_t\delta_t)=a_s\alpha_s(a_t)v(s,t)\alpha_s(\xi(t))\xi(s)\delta_{st}, 
$$
we see that $\psi$ is multiplicative if and only if  $u(s,t)\xi(st)=v(s,t)\alpha_s(\xi(t))\xi(s)$ for all $s,t\in G$.
Thus if $\xi$ with  $u=vd^{2}(\xi)$ exists, then $\psi$ given by \eqref{eq:Multiplier-isomorphism} is an isometric isomorphism. 
It is the identity on $C_0(X)$ because  
for any $a\in C_0(X)$, using that $u(1,1)=v(1,1)\xi(1)$, we get
$$
\psi(a)=\psi(au(1,1)^*\delta_1)=au(1,1)^*\xi(1) \delta_1=av(1,1)^*\delta_1 =a.
$$
Using also that  $
u(t^{-1},t) =v(t^{-1},t)\alpha_{t^{-1}}(\xi(t))\xi(1)^{*}\xi(t^{-1})$
   we get $\psi((a_t\delta_t)^*)=\psi(a_t\delta_t)^*$, for any $a_t\in C_0(X)$, $t\in G$. Thus $\psi$ preserves the involution.
Moreover, for any $\xi': G\to C_u(X)$  we have  $u=vd^{2}(\xi')$ if and only if $d^2(\xi'\xi^{-1})=1$  
if and only if $\xi'=\xi \omega$ where $\omega \in Z^1(G,C_u(X))$. Hence isomorphisms given by \eqref{eq:Multiplier-isomorphism}
are parametrised by elements of $Z^1(G,C_u(X))$. 

Let us now  assume that $\psi: F(\alpha,u) \to  F(\alpha,v)$ is any  isometric isomorphism  respecting the $G$-gradings and $\psi|_{C_0(X)}=\text{id}_{C_0(X)}$.
It suffices to  show  that $\psi$ satisfies \eqref{eq:Multiplier-isomorphism} for some $\xi:G\to C_u(X)$. 
The assumption $\psi|_{C_0(X)}=\text{id}_{C_0(X)}$ means that $\psi$ is a $C_0(X)$-bimodule map.
 For each $t\in G$, we have  a left $C_0(X)$-module map    $E_t:  F(\alpha,v) \to C_0(X)$  given by $E_t(b):=b(t)$, $b\in  F(\alpha,v)$.
Hence   the formula
$$
\xi(t)(a):=E_t\Big(\psi\big(a\delta_t\big)\Big), \qquad a\in C_0(X),
$$
defines a left $C_0(X)$-module map on  $C_0(X)$. This means that $\xi(t)$ is multiplier of $C_0(X)$.  Hence we may identify $\xi(t)$ with a function in $C_b(X)$, so that $\xi(t)(a)$ becomes $\xi(t) \cdot a$. This gives a map  $\xi:G\to C_u(X)$. 
Since $\psi$ preserves the grading, for any 
$t\in G$ and $a_t\in C_0(X)$ we have 
$
\psi(a_t\delta_t)=E_t(\psi(a_t\delta_t))\delta_t= (\xi(t) a_t) \delta_t,
$ as desired.
  \end{proof}
\begin{rem}
In light of Proposition \ref{prop:equivalent_cocycles}, twisted crossed products are in essence parameterised by the elements in the second cohomology group $H^2(G, C_u(X))$, and thus we may restrict our attention to normalised twists. Accordingly, from now on 
\begin{quote}
\emph{we assume all cocycles are normalised, that is $u(1,1)=1$ for all $u\in Z^2(G,C_u(X))$.}
\end{quote}
This   is assumed in  most sources that deal with crossed products, see \cite{Zeller-Meier}, \cite{Busby_Smith}, \cite{Packer_Raeburn}. 
Then the formulas for  involution in $F(\alpha,u)$ and the embedding of $C_0(X)$ into $F(\alpha,u)$ simplify.
\end{rem}
For any $u\in Z^2(G,C_u(X))$, the space $\ell^{1}(G,C_b(X))$ with operations \eqref{eq:crossed_product_operations} is the crossed product for the extended action  $\alpha:G\to \Aut(C_b(X))$,
which naturally embeds into the multiplier algebra $\Mult(F(\alpha,u))$ of the crossed product $F(\alpha,u)=\ell^{1}(G,C_0(X))$.  Thus both  
$C_u(X)$ and $G$  are naturally subgroups  in  the group of invertible isometries in $\Mult(F(\alpha,u))$, where 
 the $G$-embedding in question is $g\mapsto 
1_X\delta_g$.
The group  they generate is isomorphic to   $E:= C_u(X)\times G$, where $(a_{s},s)(a_{t},t):=(a_{s}\alpha_{s}(a_{t})u(s,t), st)$,   and the obvious sequence $1\to C_u(X) \to E \to G\to 1$ is exact. Equivalence classes of such group extensions of $G$ by  $C_u(X)$ correspond to elements in $H^2(G, C_u(X))$.
In particular, the class of $u$ is trivial if and only if the corresponding exact sequence is split, that is there exists a map $\xi:G\to C_u(X)$ such that 
$G\ni t \mapsto (\xi(t),t)\in   E$ is a group homomorphism. When $G=\mathbb{F}_n$ is a free group  on $n\in \N$ generators, such a homomorphism can always be constructed by universality of $\mathbb{F}_n$.
\begin{cor}\label{cor:untwist_the_twisted}
The twisted crossed product can be untwisted, in the sense that $ F(\alpha,u)\cong  F(\alpha)$ by an isometric isomorphism that respects $G$-gradings and is the identity on $C_0(X)$ if and only if
$u$ vanishes in $H^2(G, C_u(X))$. In particular, if $H^2(G, C_u(X))=0$, which always holds when $G$ is a free group, 
then all twisted crossed products by $\alpha$ are naturally isomorphic to each other.
\end{cor}
\begin{rem}
By Proposition \ref{prop:equivalent_cocycles}, isometric automorphisms of  $ F(\alpha,u)$ that are $C_0(X)$-bimodule maps respecting the $G$-gradings form a group isomorphic to 
$Z^1(G,C_u(X))$. A  \emph{$1$-coboundary} is a $1$-cocycle $\omega:G\to C_u(X)$ given by  $\omega(t):= a\alpha_t(a^{-1})$, $t\in G$,   for some $a\in C_u(X)$. Such a  cocycle
corresponds to the inner automorphism of  $ F(\alpha,u)$ implemented by $a\in C_u(X)$.
Thus the first cohomology group $H^1(G, C_u(X))$ 
can be interpreted as the outer automorphism group of the crossed product $ F(\alpha,u)$ viewed as a $G$-graded algebra.
\end{rem}
 \begin{defn}
By a \emph{representation} of a Banach algebra $B$ on a Banach space $E$ we  mean a contractive homomorphism $\pi:B\to \BB(E)$
into the algebra of bounded operators on $E$. 
We will usually assume that $\pi$ is \emph{non-degenerate} in the sense that  $\overline{\pi(B)E}=E$. We denote by $\Iso(E)\subseteq \BB(E)$ the group of invertible isometries on $E$. 
\end{defn}
It is well known that if $B$ has a contractive approximate unit and  $\pi$ is non-degenerate
then $\pi$ extends uniquely to a (unital) representation $\overline{\pi}:\M(B)\to \BB(E)$ of the multiplier algebra of $B$.  Note that $C_0(X)$  has a contractive approximate unit which is also an approximate unit in every  crossed product $F(\alpha,u)$, $u\in Z^2(G,C_u(X))$.
We fix a twisted action $(\alpha,u)$ on $C_0(X)$, and extend some classical results for representations of $F(\alpha,u)$ on Hilbert spaces from \cite{Zeller-Meier} to representations on Banach spaces.
\begin{defn}
A \emph{covariant representation} of $(\alpha,u)$ on a Banach space  $E$ is a pair $(\pi, v)$ where $\pi:C_0(X)\to \BB(E)$ is a non-degenerate representation and $v: G\to \Iso(E)$ satisfies
$$v_sv_t=\overline{\pi}(u(s,t))v_{st}, \qquad  v_t\pi(a)v_{t}^{-1}=\pi(\alpha_t(a)), \qquad\text{ for  }s,t\in G,\,\, a\in C_0(X).
$$
\end{defn}
\begin{lem}[cf. {\cite[Proposition 2.7, 2.8]{Zeller-Meier}}] \label{lem:integration_disintegration} 
Non-degenerate representations  of $F(\alpha,u)$ are in one-to-one correspondence with covariant representations  of $(\alpha,u)$.
The representation of $F(\alpha,u)$ corresponding to  $(\pi, v)$ is given by the  formula  $\pi\rtimes v\big(b\big)=\sum_{t\in G}\pi(b(t))v_t$.
\end{lem}
\begin{proof}
That  $\pi\rtimes v\big(b\big):=\sum_{t\in G}\pi(b(t))v_t$  defines a non-degenerate representation of  $F(\alpha,u)$, for a  covariant representation $(\pi, v)$, is 
 straightforward.
Let $\psi:F(\alpha,u)\to \BB(E)$ be any non-degenerate representation. Then $\pi:=\psi|_{C_0(X)}$ is a non-denenerate representation of $C_0(X)$. For any 
$t\in G$ and any contractive approximate unit $(\mu_i)$ in $C_0(X)$ the strong limit $v_t$ exists and we write $v_t= \text{s-}\lim_{i}\psi(\mu_i\delta_t)$.
Indeed, by the Cohen-Hewitt factorisation theorem, any element in  $E$ has the form $\pi(a)\xi$ for some $\xi\in E$, $a\in C_0(X)$, and
then $\lim_{i}\psi(\mu_i\delta_t)\pi(a)\xi=\lim_{i}\psi(\mu_i \alpha_t(a)\delta_t)\xi=\pi(\alpha_t(a)\delta_t)\xi$. 
This calculation also implies  that 
$v_t\pi(a)=\pi(\alpha_t(a))v_t$ and $v_t$ does not depend on the choice of an approximate unit. For $s,t\in G$ we get
\begin{align*}
v_{s}v_{t}&=\text{s-}\lim_{i}\text{s-}\lim_{j} \psi(\mu_j\delta_{s} \mu_i\delta_{t})=\text{s-}\lim_{i}\text{s-}\lim_{j}
\psi(\alpha_{s}(\mu_i)u(s,t)\mu_j\delta_{st})
\\
&=\text{s-}\lim_{i}\pi(\alpha_{s}(\mu_i)u(s,t))v_{st}=\overline{\pi}(u(s,t))v_{st}.
\end{align*}
In particular, $ v_{t^{-1}} v_{t}=\overline{\pi}(u(t^{-1},t))$ is an invertible element. Thus $v_t$ is in $\Iso(E)$, as an invertible contraction with a contractive inverse.
 Hence $(\pi, v)$ is a covariant representation. Clearly $\psi=\pi\rtimes v$ and $(\pi, v)$ is uniquely determined by this equality. 
\end{proof}
\begin{defn}
Following \cite{DDW}, \cite{Phillips}, for any class  $\RR$ of representations of $F(\alpha,u)$ we denote by $F_{\RR}(\alpha,u)$ the Hausdorff completion of $F(\alpha,u)$ in  the seminorm given by 
$$\|b\|_{\RR}:=\sup\{\|\psi(b)\|:  \psi\in \RR\}.
$$ 
\end{defn}
\begin{defn}
We say that a Banach space $E$ is an \emph{$L^p$-space}, for $p\in[1,\infty]$, if there is an isometric isomorphism $E\cong L^{p}(\mu)$ for some Lebesgue space $L^{p}(\mu)$ on some measure space $(\Omega,\FF,\mu)$. We say that a Banach algebra $B$ is an \emph{$L^p$-operator algebra} if there is an isometric representation $B\to \BB(E)$ into bounded operators on some $L^p$-space $E$. 
Similarly, we define \emph{$C_0$-spaces} as Banach spaces isometrically isomorphic to $C_0(\Omega)$ for some topological space $\Omega$, 
and \emph{$C_0$-operator algebras} as those that can be embedded into $B(C_0(\Omega))$. 
\end{defn}
We have natural injective representations of $F(\alpha,u)$ on $L^p$-spaces (with counting measures) and $C_0$-spaces (of functions on  discrete spaces).
Recall that $\varphi$ is the action of $G$ on $X$ dual to the action  $\alpha$ of $G$ on $C_0(X)$, and $u\in Z^2(G,C_u(X))$ is a fixed (normalised) twist.
\begin{lem}\label{lem:regular_representation}
For each $p\in[1,\infty)$  the formula, for  $\xi \in \ell^p(X\times G)$, $b\in F(\alpha,u)$,
$$
\Lambda_p(b)\xi(x,t)=\sum_{s\in G} b(s)(x)u(s,s^{-1}t)(x)\xi(\varphi_{s^{-1}}(x),s^{-1}t)
$$
defines an injective non-degenerate representation $\Lambda_p:F(\alpha,u)\to B(\ell^p(X\times G))$.
The same formula defines an injective non-degenerate representation $\Lambda_{\infty}:F(\alpha,u)\to B(c_0(X\times G))$.
\end{lem}
\begin{proof} Routine calculations show that the formulas, for $a\in C_0(X)$, $s,t\in G$ and $x\in X$,
$$
\pi(a)\xi(x,t)=a(x)\xi(x,t), \quad v_s \xi(x,t)=u(s,s^{-1}t)(x)\xi(\varphi_{s^{-1}}(x),s^{-1}t)
$$
define a covariant representation of $(\alpha,u)$ on $\ell^p(X\times G)$ or $c_0(X\times G)$. For instance, 
\begin{align*}
v_s v_t\xi(x,r)&=u(s,s^{-1}r)(x)u(t,(st)^{-1}r)(\varphi_{s^{-1}}(x))\xi(\varphi_{(st)^{-1}}(x),(st)^{-1}r) 
\\
&\stackrel{\eqref{eq:cocycle_identity}}{=}u(s,t)(x) u(st,(st)^{-1}r)(\varphi_{(st)^{-1}}(x))\xi(\varphi_{(st)^{-1}}(x),(st)^{-1}r) =\overline{\pi}(u(s,t))v_{st}\xi(x,r).
\end{align*}
Clearly, $\Lambda_p=\pi\rtimes v$. To see that $\Lambda_p$ is injective, let $\{\mathds{1}_{x,t}\}_{(x,t)\in X\times G}$ be the standard Schauder basis in $\ell^p(X\times G)$, 
and note that, for $a\in C_0(X)$, we have
$$\pi(a)\mathds{1}_{x,t}= a(x) \mathds{1}_{x,t} \quad \text{ and }\quad  v_s \mathds{1}_{x,t}=u(s,t)(\varphi_s(x))\mathds{1}_{\varphi_s(x),st}.
$$ 
If  $\Lambda_p(\sum\limits_{s\in G} a_s\delta_s)=0$, then for all $(x,t)\in X\times G$  we have
$$
\Lambda_{p}\Big(\sum\limits_{s\in G} a_s\delta_s\Big)\mathds{1}_{x,t}=\sum\limits_{s\in G} a_s(\varphi_s(x))u(s,t)(\varphi_s(x))\mathds{1}_{\varphi_s(x),st}=0.
$$
Since the elements $\{\mathds{1}_{\varphi(x),st}\}_{s\in G}$ are linearly independent and $u(s,t)(\varphi_s(x))$ are non-zero numbers, this implies 
 that $a_s(\varphi_s(x))=0$ for all $s\in G$ and $x\in X$. Hence $\sum\limits_{s\in G}a_s\delta_s=0$.
\end{proof}
\begin{defn}\label{def:L^p_crossed_products}
For $p\in [1,\infty)$ we define the (full) \emph{$L^p$-operator algebra crossed product} $F^{p}(\alpha,u)$ to be $F_{\RR}(\alpha,u)$ where $\RR$ is the class of all non-degenerate representations on  $L^p$-spaces. Thus $F^{p}(\alpha,u)$ is the completion of $F(\alpha,u)$ in the norm
$$
\|b\|_{L^p}=\{\|\psi(b)\|: \psi \text{ is a non-degenerate representation of $F(\alpha,u)$ on an $L^p$-space}\}.
$$
The \emph{reduced $L^p$-operator algebra crossed product} is $F^{p}_{\red}(\alpha,u):=F_{\{\Lambda_p\}}(\alpha,u)$.
Thus we may  identify $F^p_{\red}(\alpha,u)$ with $\overline{\Lambda_p(F(\alpha,u))}\subseteq B(\ell^{p}(X\times G))$.
Similarly, we define $F^{\infty}(\alpha,u)$ where $\RR$ is the class of all non-degenerate representations on $C_0$-spaces, 
and $F^{\infty}_{\red}(\alpha,u):=F_{\{\Lambda_{\infty}\}}(\alpha,u)$.
\end{defn}

\begin{rem}\label{rem:consequences_of_the_groupoid_paper}
 It follows from the results of \cite{BKM1} (using the groupoid model for the above crossed products) that we get the same norms and Banach algebras allowing degenerate representations, and then
$F^{\infty}(\alpha,u)$ is also universal for representations on $L^{\infty}$-spaces. 
For $p=2$ the full and reduced crossed products coincide with the standard $C^*$-algebraic crossed products  $C^*(\alpha,u)$ and $C^*_{\red}(\alpha,u)$. It also follows from \cite{BKM1} that for $b\in F(\alpha,u)$ we have
\begin{equation}\label{eq:norms_for_1_and_infty}
 \|b\|_{L^1}=\|\Lambda_1(b)\|=\max_{x\in X} \sum_{t\in G}|b(t)(x)|,\quad  \|b\|_{L^\infty}=\|\Lambda_{\infty}(b)\|=\max_{x\in X} \sum_{t\in G}|b(t)(\varphi_{t}(x))|,
\end{equation}
and if $1/p+1/q=1$, then $\|b\|_{L^p}\leq \|b\|_{L^1}^{\frac{1}{p}}\|b\|_{L^{\infty}}^{\frac{1}{q}}$ and $F^{p} (\alpha,u)\stackrel{anti}{\cong}F^{q} (\alpha,u)$ and $F^{p}_{\red} (\alpha,u)\stackrel{anti}{\cong}F^{q}_{\red} (\alpha,u)$ where
the isometric anti-isomorphisms are induced by the involution in  $F(\alpha,u)$. Thus we have the following hierarchy of the introduced crossed products each of which is a  completion of  $F(\alpha,u)$:

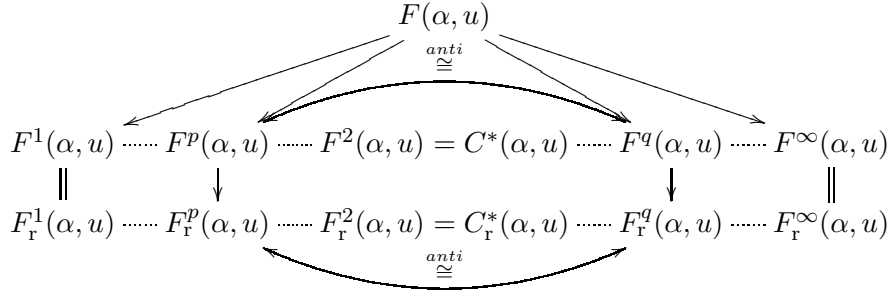
$$
\xymatrixcolsep{1pc} \xymatrixrowsep{1pc} 
\xymatrix{
 &  &  F(\alpha,u)  \ar@{->}[lldd]  \ar@{->}[ldd] \ar@{->}[ddr]  \ar@{->}[ddrr]&  &  
\\
 &  &  &  &
\\  
F^1(\alpha,u) \ar@{=}[d] \ar@{.}[r] & F^{p} (\alpha,u) \ar@{->}[d]  \ar@{.}[r]   \ar@/^2pc/[rr]^{\stackrel{anti}{\cong}} & F^{2} (\alpha,u)= C^*(\alpha,u)   \ar@{.}[r] 
 & F^{q} (\alpha,u) \ar@/_2pc/[ll]  \ar@{->}[d] \ar@{.}[r] & 
 F^\infty(\alpha,u)  \ar@{=}[d] 
\\
F^1_{\operatorname {r}}(\alpha,u) \ar@{.}[r] & F^{p}_{\operatorname {r}} (\alpha,u)  \ar@{.}[r]   \ar@/_2pc/[rr]  & F^2_{\operatorname {r}} (\alpha,u)= C^*_{\operatorname {r}}(\alpha,u)   \ar@{.}[r]  &  F^{q}_{\operatorname {r}} (\alpha,u)   \ar@/^2pc/[ll]_{ \stackrel{anti}{\cong}} \ar@{.}[r] & 
 F^\infty_{\operatorname {r}}(\alpha,u)  .
}
$$ 
Here the downward arrows are representations which are identities on $F(\alpha,u)$, the algebras in the middle column are Banach $*$-algebras and the horizontal anti-isomorphims are given by
the involution in $F(\alpha,u)$.
\end{rem}
We also mention a generalisation of the trivial representation of a group that only  works in the untwisted case. It  plays a crucial role 
in characterisation of the generalised intersection property for untwisted crossed products (see Theorem \ref{thm:generalised_intersection_property} below),
and indirectly it is also used in the proof of the corresponding theorem for  twisted crossed products (Theorem \ref{thm:twisted_generalised_intersection_property}).

\begin{ex}\label{ex:trivial_rep}
The formulas 
$
\pi(a)\xi(x):=a(x)\xi(x)$ and  $v_t \xi(x):=\xi(\varphi_{t^{-1}}(x))
$
define a covariant representation of $\alpha$ on each of the spaces $\ell^{p}(X)$, $p\in[1,\infty)$, and $c_0(X)$. We denote their integrated forms by  $\Lambda_{p}^{\tr}:F^p(\alpha)\to B(\ell^{p}(X))$, $\Lambda_{\infty}^{\tr}:F^p(\alpha)\to B(c_0(X))$ and call them \emph{$C_0(X)$-trivial representations} of $F(\alpha)$. Explicitly,
\begin{equation}\label{eq:trivial_rep}
\Lambda_p^{\tr}(b)\xi(x)=\sum_{s\in G} b(s)(x)\xi(\varphi_{s^{-1}}(x)).
\end{equation}
\end{ex}

\section{Reduced $L^p$-crossed products and amenable actions}
Definition \ref{def:L^p_crossed_products} agrees with notions introduced by Phillips in \cite{Phillips} for untwisted crossed products, which were  extended to the non-separable setting in \cite{cgt} and \cite{Gardella}. 
Reduced crossed products in these sources are defined using induced representations. We explain the equivalence between the two approaches. Let $(\alpha,u)$ be a twisted action  on $C_0(X)$.
For any non-degenerate representation $\pi:C_0(X)\to \BB(E)$ on a Banach space $E$ we have
a covariant representation of $(\alpha, u)$ on $\ell^{p}(G, E)$, $p\in [1,\infty)$, given by 
$$
\widetilde{\pi}(a)\xi(t):=\pi(\alpha_{t}^{-1}(a))\xi (t),\quad  \widetilde{v}_s \xi(t):=\overline{\pi}(\alpha_{t}^{-1}(u(s,s^{-1}t)))\xi(s^{-1}t),
$$
$a\in C_0(X)$, $\xi\in  \ell^{p}(G, E)$. We call $\Ind_p(\pi):=\widetilde{\pi}\rtimes \widetilde{v}: F(\alpha,u)\to B(\ell^{p}(G, E))$ the \emph{$\ell^p$-representation induced by $\pi$}. 
Explicitly, for $b\in F(\alpha,u)$ we have
$$
\Ind_p(\pi)(b)\xi (t)=\sum_{s\in G} \pi(\alpha_{t}^{-1}(b(s)u(s,s^{-1}t)))\xi(s^{-1}t).
$$
Note that if $m_p:C_0(X)\to B(\ell^p(X))$ is given by multiplication: $m_p(a)\xi (x):=a(x)\xi(x)$, then the representations $\Ind_p(m_p)$ and $\Lambda_p$ are equivalent (conjugated by 
the invertible isometry $U\in B(\ell^p(G,\ell^{p}(X))$ given by $U\xi(t)(x):=\xi(t)(\varphi_t(x))$).
Similarly,  one may define $\Ind_{\infty}(\pi)=\pi\rtimes v: F(\alpha,u)\to B(c_0(G, E))$ and then $\Ind_{\infty}(m_{\infty})\cong\Lambda_{\infty}$
where $m_{\infty}:C_0(X)\to B(c_0(X))$ is given by $m_{\infty}(a)\xi (x):=a(x)\xi(x)$. 
We now generalise \cite[Proposition 6.4]{cgt} where $X$ is compact, $p\in[1,\infty)\setminus\{2\}$, and there is no twist.
\begin{prop}\label{prop:reduced_norm_is_induced}
 For any $p\in[1,\infty)$ and $b\in F(\alpha,u)$ we have
$$
\|\Lambda_p(b)\|=\max\{\|\Ind_p(\pi)(b)\|: \pi \text{ is a non-degenerate representation on an }L^p\text{-space}\},
$$
and $\|\Lambda_{\infty}(b)\|=\max\{\|\Ind_{\infty}(\pi)(b)\|: \pi \text{ is a non-degenerate representation on a }C_0\text{-space}\}$.
\end{prop}
\begin{proof}  Let $b\in F(\alpha,u)$.
 Since $\Lambda_p\cong \Ind_p(m_p)$ we only need to show that  $\|\Ind_p(\pi)(b)\| \leq \|\Lambda_p(b)\|$ for any non-degenerate $\pi$ on an $L^p$-space or $C_0$-space when $p=\infty$. Let us first consider the case when $p=\infty$, and let $\pi:C_0(X)\to B(C_0(\Omega))$ be a non-degenerate representation where $\Omega$ is some locally compact Hausdorff space. By \cite[Theorem 2.16]{BKM1}  there is a continuous map $\Phi:\Omega\to X$ such that  $\pi(a)\xi (y)=a(\Phi(y))\xi(y)$ for $\xi \in C_0(\Omega)$,  $y\in \Omega$ and $a\in C_0(X)$.
 Using this and \eqref{eq:norms_for_1_and_infty}, for $\xi\in C_0(G, C_0(\Omega))$,  we get
\begin{align*}
\| \Ind_{\infty}(\pi)(b)\xi\|&= \max_{y\in \Omega, t\in G} \left|\sum_{s\in G}[b(s)u(s,s^{-1}t)](\varphi_{t}(\Phi(y))) \xi(s^{-1}t)(y)\right| 
\\
&\leq \max_{x\in X} \sum_{s\in G} \left| [b(s)u(s,s^{-1}t)](\varphi_t(x))\right| \max_{y\in \Omega, t\in G} \left| \xi(t)(y)\right|=
\|\Lambda_{\infty}(b)\| \|\xi\|,
\end{align*}
and so $\|\Ind_{\infty}(\pi)(b)\| \leq \|\Lambda_{\infty}(b)\|$. 
For $p=2$ the assertion  is a well known fact, see \cite[Theorem 4.11]{Zeller-Meier}. 
Hence it suffices to consider the case when  $p\in [1,\infty)\setminus \{2\}$. 
Let us first consider a probability Borel  measure $\mu$ on $X$ and a representation $m_p^{\mu}:C_0(X)\to B(L^p(\mu))$ 
given by multiplication $m_p^{\mu}(a)\xi (x):=a(x)\xi(x)$, $\xi\in L^p(\mu)$. 
Denote by 
$\Ind_p(\mu)=\Ind_p(m_p^{\mu})$ the corresponding representation of $F(\alpha,u)$. Then $\Ind_p(\delta_x)$, where $\delta_x$ is the point mass measure at $x\in X$, can be treated as a
representation on $\ell^{p}(G)$. Moreover,  $\Lambda_p\cong\Ind_p(m_p)\cong \bigoplus_{x\in X/G} \Ind_p(\delta_x)$ where $X/G$ is the orbit space for the action $\varphi$.
Hence for any $\xi\in \ell^{p}(X,L^p(\mu))$,  putting $\xi^x(t)=\xi(t)(x)$   for  $x\in X$, we get 
\begin{align*}
\| \Ind_p(\mu)(b)\xi\|^p&= \int_{X}\sum_{t\in G} \left|\sum_{s\in G}[b(s)u(s,s^{-1}t)](\varphi_{t}(x))\xi(s^{-1}t)(x)\right|^{p} d\mu
\\
&=  \int_{X} \|\Ind_p(\delta_x)(b)\xi^x\|_{\ell^{p}(G)}^p\, d\mu
\\
&\leq \sup_{x\in X} \|\Ind_p(\delta_x)(b)\|^p \int_{X} \|\xi^x\|_{\ell^{p}(G)}^p\, d\mu =\|\Lambda_p(b)\|^p \|\xi\|^p.
\end{align*}
Thus $\| \Ind_p(\mu)(b)\|\leq \|\Lambda_p(b)\|$. Let now $\pi:C_0(X)\to B(L^{p}(\nu))$ be any non-degenerate representation on some measure space $(Y,\nu)$.
By passing to the extension of the action $\alpha$ and representation $\pi$ to the algebra $\C1+C_0(X)\subseteq C_{b}(X)$ we may assume that $X$ is compact. 
For every $\varepsilon>0$ there is $\xi \in \ell^{p}(G,L^{p}(\nu))$  such that
$$ 
\|\xi\|=1\,\, \text{  and }\,\, \|\Ind_p(\pi)(b)\|-\varepsilon\leq \|\Ind_p(\pi)(b)\xi \|.
$$
Since the set $\{t\in G: \xi(t)\neq 0\}$ is countable, it generates a countable subgroup $G'$ of $G$. 
Since $\{s\in G: b(s)\neq 0\}$ is countable,  the $C^*$-subalgebra $A$ of $C(X)$ generated by 
$ \{\alpha_t(b(s)): s\in G,t\in G'\}\cup \{u(s,t): s,t\in G'\}$ is separable, unital and  the twisted action $(\alpha,u)$ of $G$ on $C(X)$ restricts to a twisted action $(\alpha', u')$ of $G'$ on $A$. Write  $A\cong C(X')$ for a  compact Hausdorff $X'$, and note that $\pi$ restricts to a unital representation 
$\pi:C(X')\to B( L^{p}(\nu))$. By the proof of \cite[Proposition 1.25]{Phillips}
 there is a separable Banach subspace $F\subseteq L^{p}(\nu)$   such that $F$ is an $L^p$-space, $\pi(C(X'))F\subseteq F$ and  $\{\xi(t):t\in G\}\subseteq F$. 
As $F$ is a separable $L^p$-space, there is  a standard Borel probability measure $\nu'$ such that $F\cong L^p(\nu')$. We identify $F=L^p(\nu')$. 
Then  $\pi$ yields a unital representation    $\pi':C(X')\to B(L^{p}(\nu'))$. By construction $\xi\in \ell^{p}(G',L^{p}(\nu'))\subseteq \ell^{p}(G,L^{p}(\nu))$, $b\in F(\alpha',u')$ and
$\Ind_{p}(\pi)(b)\xi=\Ind_{p}(\pi')(b)\xi$. 
Let $\mu$ be a Borel probability measure on $X'$ such that $\mu(a)=\nu'(\pi'(a))$ for $a\in C(X')$.
By (the proof of) \cite[Lemma 6.3]{cgt} we have $\|\Ind_p(\pi')(b)\|=\|\Ind_p(\mu)(b)\|$.
Hence
$$
\|\Ind_{p}(\pi)(b)\xi\|=\|\Ind_{p}(\pi')(b)\xi\|\leq  \| \Ind_p(\mu)(b)\|\leq \|\Lambda_p'(b)\|
$$
where $\Lambda_p':F(\alpha',u')\to \ell^{p}(G',L^{p}(\nu'))$ is the regular representation (see the previous step).
Since  $\|\Lambda_p'(b)\|=\sup_{x'\in X'/G'} \| \Ind_p(\delta_{x'})\|$ there is $x' \in X'$ with 
$\|\Lambda_p'(b)\|-\varepsilon \leq \| \Ind_p(\delta_{x'})(b)\|$. Taking any preimage $x\in X$ of $x'$ under the quotient map $X\onto X'$,
one verifies that $\|\Ind_p(\delta_{x'})(b)\|=\|\Ind_p(\delta_{x})(b)\|$. Consequently, gathering all the pieces together, we get
$$
\|\Ind_p(\pi)(b)\|-\varepsilon\leq \|\Ind_p(\pi)(b)\xi \|\leq \|\Lambda_p'(b)\|\leq \|\Lambda_p(b)\| + \varepsilon,
$$ 
which implies that $\|\Ind_p(\pi)(b)\|\leq \|\Lambda_p(b)\|$.
\end{proof}

\begin{defn}
For any $p\in[1,\infty]$, the representation from Lemma \ref{lem:regular_representation} extends to a representation
$\Lambda_p:F^{p}(\alpha,u)\to F^p_{\red}(\alpha,u)$ that we call the \emph{regular representation}. 
\end{defn}
 By Remark \ref{rem:consequences_of_the_groupoid_paper},  
$\Lambda_1$ and $\Lambda_{\infty}$ are identity maps (isometric isomorphisms) but for $p\in (1,\infty)$ the regular representation $\Lambda_p:F^{p}(\alpha,u)\to F^p_{\red}(\alpha,u)$ is not injective in general.
Here we give a detailed proof that all $\Lambda_p$'s are isometric for amenable actions. This was stated without a proof for non-twisted actions on compact spaces in \cite[Lemma 7.5]{cgt}.
 
\begin{defn}[{\cite[Proposition 2.2(3)]{Anantharaman-Delaroche}}] An action  $\varphi:G\to \Homeo(X)$ is \emph{amenable}  if  there exists an \emph{approximate invariant mean}, that is a
net $(f_i)_{i\in I}$ of 
finitely supported non-negative functions $f_i:G\to  C_c(X)^+$ such that  the net $\sum_{t\in G} f_i(t)$ converges compactly to $1$ and the net $\sum_{t\in G} |\alpha_{s}(f_i(t))- f_{i}(st)|$  converges compactly  to $0$, for every $s\in G$   (\emph{compact convergence} means uniform convergence on compact subsets).
\end{defn}
\begin{lem}\label{lem:p_invariant_means} 
Let $p,q\in (1,\infty)$ with $1/p+1/q=1$. 
If the action $\varphi$ is amenable, then there are nets  $(g_i)_{i\in I}$, $(h_i)_{i\in I} \subseteq C_c(G,C_c(X)^+)$ such that 
each of the  nets 
$\sum_{t\in G} g_i(t)^p$, $\sum_{t\in G} h_i(t)^q$  and  $\sum_{t\in G} h_{i}(st)\alpha_{s}\big(g_i(t)\big)
$, $s\in G$, 
converge compactly to $1$.
\end{lem}
\begin{proof}
If $(f_i)_{i\in I}$ is an approximate invariant mean, then $g_i(t):=f_i(t)^{\frac{1}{p}}$ and $h_i(t)=f_i(t)^{\frac{1}{q}}$
are the desired nets. Indeed, using H\"olders inequality, and that $|a-b|^p\leq |a^p-b^{p}|$ for $a,b\geq 0$, for any $s\in G$   we get
\begin{align*}
\left| \sum_{t\in G} \alpha_{s}(g_i(t))h_{i}(st) -\sum_{t\in G} f_i(t)\right|
&\leq 
\sum_{t\in G} \left|\alpha_{s}(f_i(t)^{\frac{1}{p}})f_{i}(st)^{\frac{1}{q}} - f_i(st)\right|
\\
&=
\sum_{t\in G} \left|\alpha_{s}(f_i(t)^{\frac{1}{p}}) -f_i(st)^{\frac{1}{p}}\right| f_{i}(st)^{\frac{1}{q}}
\\
&\leq\left(\sum_{t\in G}\left| \alpha_{s}(f_i(t)^{\frac{1}{p}}) -f_i(st)^{\frac{1}{p}}\right|^p\right)^{\frac{1}{p}}   \left(\sum_{t\in G} f_{i}(st) \right)^{\frac{1}{q}}
\\
&\leq \left(\sum_{t\in G} \left|\alpha_{s}(f_i(t)) -f_i(st)\right|\right)^{\frac{1}{p}}\left(\sum_{t\in G} f_{i}(t) \right)^{\frac{1}{q}}. 
\end{align*}
As $\sum_{t\in G} f_{i}(t) $ converges compactly to $1$, the above estimate implies that $\sum_{t\in G} \alpha_{s}(g_i(t))h_{i}(st)$ also  converges compactly to $1$.
\end{proof}
\begin{thm}\label{thm:ambenability_reduced_equals_full}
If the action $\varphi$ is amenable, then for any $p\in (1,\infty)$ the regular representation $\Lambda_p:F^{p}(\alpha,u)\to F^p_{\red}(\alpha,u)$ is isometric.
\end{thm}
\begin{proof}  Let $(\pi,v)$ be a covariant representation of $(\alpha,u)$ on some Banach space $E$ and let $(g_i)_{i\in I}$, $(h_i)_{i\in I}$ be nets in  $C_c(G,C_c(X)^+)$ as in Lemma \ref{lem:p_invariant_means}.
For $\xi \in E$, $\eta \in E'$ we put  
 $$
\xi_i(t):= v_{t^{-1}}\pi(\overline{u(t,t^{-1})}g_i(t))\xi, \qquad  \eta_i(t):=v_{t}' \pi(h_i(t))'\eta, \qquad t\in G.
$$
Then $\xi_i\in \ell^{p}(G,E)$ and $\eta_i\in \ell^{q}(G,E')$. Using the pairing establishing the isomorphism $\ell^{p}(G,E)'\cong \ell^{q}(G,E')$,
 for any $a\in C_c(X)$ and $s\in G$ we get
\begin{align*}
\langle \Ind_p(\pi)(a\delta_{s})\xi_{i}, \eta_{i}\rangle 
&= \sum_{t\in G} \langle \pi(\alpha_{st}^{-1}(au(s,t)
)) \xi_{i}(t),\eta_i(st)\rangle
\\
&
= \sum_{t\in G} \Big\langle \pi(\alpha_{st}^{-1}(au(s,t)
)) v_{t^{-1}} \pi\big(\overline{u(t,t^{-1})} g_i(t)\big)\xi, v_{st} '\pi(h_i(st))'\eta\Big\rangle
\\
&
=  \sum_{t\in G}  \Big\langle \pi\Big(h_{i}(st) au(s,t) u(st,t^{-1})  \alpha_{s}(g_i(t))  \overline{u(t,t^{-1})}\Big)v_{s} \xi , \eta\Big\rangle 
\\
&
=\Big\langle \pi\Big( \sum_{t\in G} h_{i}(st) \alpha_{s}(g_i(t)) a\Big)v_{s} \xi , \eta\Big\rangle 
\longrightarrow \langle \pi(a)v_{s} \xi,\eta\rangle.
\end{align*}
By linearity, for every $b\in C_c(G,C_c(X))\subseteq F(\alpha,u)$ we get $ \langle \pi\rtimes v (b) \xi,\eta\rangle=\lim_{i} 
\langle \Ind_p(\pi)(b)\xi_{i}, \eta_{i}\rangle$.
This shows that $\pi\rtimes v (b)\neq 0$ implies  $\Ind_p(\pi)(b)\neq 0$ for any $b\in F(\alpha,u)$. To get norm estimates we specialise to the case where
$E=L^p(\mu)$ for some measure $\mu$ (then $E'=L^{q}(\mu)$).  Without changing the Banach space we may assume that $\mu$ is localisable.
$E=L^p(\mu)$ for some measure $\mu$ (then $E'=L^{q}(\mu)$).  Without changing the Banach space we may assume that $\mu$ is localisable.
Also we may assume that $p\neq 2$, as for $p=2$ the assertion is well known, see \cite[Theorem 5.3]{Anantharaman-Delaroche}.
 Then $\pi$ is given by multiplication operators, so there is a homomorphism $\pi_0:C_0(X)\to L^{\infty}(\mu)$, 
such that $\pi(a)\xi=\pi_0(a) \cdot \xi$, $\xi \in L^p(\mu)$, see \cite[Theorem 2.14]{BKM1}. Using this, and that $\sum_{t\in G} g_i(t)^p$ converges compactly to $1$, 
for any $\xi \in \pi(C_c(X))L^{p}(\mu)$, we get  
\begin{align*}
\|\xi_i\|^p&=\sum_{t\in G}\|v_{t^{-1}}\pi(g_i(t))\xi\|^p=\sum_{t\in G}\int |\pi_0(g_i(t))\xi|^p \, d\mu=\int \pi_0\left(\sum_{t\in G}g_i(t)^p\right)| \xi|^p \, d\mu
\\
&
\longrightarrow  \int |\xi|^p \, d\mu=\|\xi\|^p.
\end{align*}
Similarly, for $\eta \in \pi(C_c(X))L^{q}(\mu)$ we get $\|\eta_i\|\longrightarrow  \|\eta\|$.
Thus for any $b\in C_c(G,C_c(X))\subseteq F(\alpha,u)$ we have
 $ |\langle \pi\rtimes v (b) \xi,\eta\rangle|=|\lim_{i} 
\langle \Ind_p(\pi)(b)\xi_{i}, \eta_{i}\rangle|\leq  \|\Ind_p(\pi)(b)\| \|\xi\| \|\eta\|$,
which implies $\| \pi\rtimes v (b)\|\leq \|\Ind_p(\pi)(b)\|$.
Since $(\pi, v)$ was arbitrary, we obtain  $\|b\|_{L^p}\leq \|\Lambda_p(b)\|$ by  Proposition \ref{prop:reduced_norm_is_induced}.
Therefore the full and reduced norms coincide.
\end{proof}

\section{Fourier decomposition and reduced crossed products}
\label{sec:reduced crossed products}
 Throughout this section $(\alpha,u)$ is a fixed twisted action  on $C_0(X)$.  For each $t\in G$, the evaluation map $
E_{t}(b):=b(t)
$   defines a contractive linear map $E_{t}:F(\alpha,u)\to C_0(X)$, and $b\in F(\alpha,u)$ is 
determined by the coefficients $\{E_{t}(b)\}_{t\in G}\subseteq C_0(X)$.  We formalise this structure in other Banach algebra completions of $F(\alpha,u)$
as follows.
\begin{defn} 
Let $\RR$ be a family of representations that separates the points of $F(\alpha,u)$.
 We say that  $F_{\RR}(\alpha,u)$ is a \emph{Banach algebra crossed product for $(\alpha,u)$} or that $F_{\RR}(\alpha,u)$ \emph{admits the Fourier decomposition}  if $E_{1}:F(\alpha,u)\to C_0(X)$ extends to a contractive map $E_{1}^{\RR}:F_{\RR}(\alpha,u)\to C_0(X)$. 
\end{defn}
By an \emph{ideal} in a Banach algebra we always mean a closed two-sided ideal.
\begin{lem}\label{lem:Fourier_decomposition_maps}
If $F_{\RR}(\alpha,u)$ is a Banach algebra crossed product, then each $E_{t}:F(\alpha,u)\to C_0(X)$ extends to a contractive  map $E_{t}^{\RR}:F_{\RR}(\alpha,u)\to C_0(X)$, $t\in G$.
They are left $C_0(X)$-module maps and their joint kernel $\bigcap_{t\in G}\ker(E_t^{\RR})$ is an ideal in $F_{\RR}(\alpha,u)$, which has a zero intersection with $C_0(X)$.
\end{lem}
\begin{proof}
We may assume that $F_{\RR}(\alpha,u)\subseteq \BB(E)$ acts in a non-degenerate way on some Banach space $E$ (one may take $E=F_{\RR}(\alpha,u)$).
By Lemma \ref{lem:integration_disintegration}, the inclusion map disintegrates to a covariant representation $(\text{id}_{C_0(X)}, v)$.
In particular,  $b=\sum_{t\in G} b(t)v_t$ for $b\in  F(\alpha,u)\subseteq F_{\RR}(\alpha,u)$. The invertible isometries $v_t$, $t\in G$, 
are multipliers of $C_c(G,C_0(X))\subseteq F(\alpha,u)\subseteq F_{\RR}(\alpha,u)$,
and so also of $F_{\RR}(\alpha,u)$. Thus  the formula $E_{t}^{\RR}(b):=E_{1}^{\RR}(b v_{t}^{-1})$ yields a well-defined  contractive map $E_{t}^{\RR}:F_{\RR}(\alpha,u)\to C_0(X)$. Simple verification, using that $v_{t}^{-1}=\overline{u(t^{-1},t)} v_{t^{-1}}= v_{t^{-1}}\overline{u(t,t^{-1})}$, shows that $E_{t}^{\RR}$ extends  $E_t$ and we also have 
$
E_{t}^{\RR}(b)=\alpha_t\left(E_{1}^{\RR}(v_{t}^{-1}b)\right)
$, $b\in F_{\RR}(\alpha,u)$.
Since $E_{1}^{\RR}$ is  $C_0(X)$-bimodule map, this readily implies that 
\begin{equation}\label{eq:twisted_module_maps}
E_{t}^{\RR}(ab)=a E_{t}^{\RR}(b), \qquad E_{t}^{\RR}(ba)=E_{t}^{\RR}(b)\alpha_t(a),\qquad a\in C_0(X),\, b\in F_{\RR}(\alpha,u).
\end{equation}
Similarly, using that $v_{t}^{-1}v_{s}^{-1}=v_{st}^{-1}\overline{u(s,t)}$ for any $s,t\in G$ and $b\in  F_{\RR}(\alpha,u)$ we  have  
$$E_{t}^{\RR}(v_s^{-1}b)=E_{st}^{\RR}(b) \alpha_t \big(\overline{u(s,t)}\big), \qquad  E_{s}^{\RR}(bv_t^{-1})=E_{st}^{\RR}(b) \overline{u(s,t)}.
$$
Exploiting these relations we see that $\bigcap_{t\in G}\ker(E_t^{\RR})$ is an ideal. 
\end{proof}
\begin{rem}\label{rem:Banach_algebra_with_Fourier_decomposition}
A Banach algebra completion  $B$ of $F(\alpha,u)$ is of the form $F_{\RR}(\alpha,u)$  if and only if   $\|b\|_{B}\leq \sum_{t\in G} \|b(t)\|_{\infty}$ for any $b\in F(\alpha,u)$.
If this holds, then
\begin{align*}
B\text{ is a crossed product}\,\, &\Longleftrightarrow\,\, \|b(1)\|_{\infty}\leq \|b\|_{B}\text{ for all } b\in F(\alpha,u) 
\\
 &\stackrel{\ref{lem:Fourier_decomposition_maps}}{\Longleftrightarrow}\,\, \max_{t\in G}\|b(t)\|_{\infty}\leq \|b\|_{B}\text{ for all } b\in F(\alpha,u)
\\
 &\Longleftrightarrow\,\, \text{the inclusion $F(\alpha,u)\subseteq C_0(G,C_0(X))$ extends to a}
\\
&  \qquad \,\, \text{ contractive linear map $j_{B}:B\to  C_0(G,C_0(X))$}. 
\end{align*} 

If the above holds, then
the maps $E_{t}^{B}:B\to C_0(X)$, $t\in G$, from Lemma \ref{lem:Fourier_decomposition_maps}, are compositions of $j_B$ and evaluations on $C_0(G,C_0(X))$. 
In particular,   $\bigcap_{t\in G}\ker(E_t^{B})=\ker j_B$.
If $G$ is finite, then   $B=F(\alpha,u)$ as topological algebras, becasue for $b\in F(\alpha,u)$ we have 
$\max_{t\in G}\|b(t)\|_{\infty} \leq \|b\|_{B}\leq \sum_{t\in G} \|b(t)\|_{\infty} \leq |G| \max_{t\in G}\|b(t)\|_{\infty}
$, and so the norms are equivalent.
\end{rem}

\begin{defn}\label{def:reduced_crossed_product} Let $F_{\RR}(\alpha,u)$ be a Banach algebra crossed product. We call the maps $\{E_{t}^{\RR}\}_{t\in G}$ in Lemma \ref{lem:Fourier_decomposition_maps}
the \emph{Fourier decomposition} (\emph{FD} for short) for $F_{\RR}(\alpha,u)$. We say that  $F_{\RR}(\alpha,u)$ is a \emph{reduced crossed product} if
$\bigcap_{t\in G}\ker(E_t^{\RR})=\{0\}$, equivalently every $b\in F_{\RR}(\alpha,u)$ is uniquely
determined by the coefficients $\{E_{t}^{\RR}(b)\}_{t\in G}\subseteq C_0(X)$. 
\end{defn}
\begin{rem}
Obviously, $F(\alpha,u)$ is a reduced crossed product. By Remark \ref{rem:Banach_algebra_with_Fourier_decomposition}, a crossed product $F_{\RR}(\alpha,u)$ is  reduced if and only if  the inclusion $F(\alpha,u)\subseteq C_0(G,C_0(X))$ extends to a contractive injective linear map $j_{\RR}:F_{\RR}(\alpha,u)\to  C_0(G,C_0(X))$, and this is automatic when $G$ is finite. Every crossed product $F_{\RR}(\alpha,u)$  can be reduced in the sense that the quotient of 
$F_{\RR}(\alpha,u)$ by $\bigcap_{t\in G}\ker(E_t^{\RR})$ is naturally a reduced crossed product.
\end{rem}
\begin{rem} When $X=\{x\}$ and there is no twist, our definition agrees with the notion of an  abstract reduced Banach group algebra that  Phillips introduced in \cite[Definition 1.1]{Phillips19}. 
\end{rem}
\begin{rem} Exotic $C^*$-algebra crossed products for $(\alpha,u)$,  see \cite{Buss_Echt_Will}, are exactly those  crossed products $F_{\RR}(\alpha,u)$ which are $C^*$-algebras. 
Up to an isometric isomorphism there is exactly one reduced crossed product  for $(\alpha,u)$ which is a $C^*$-algebra, and this is the standard reduced $C^*$-algebra crossed product $C_{\red}^*(\alpha,u)$.
\end{rem}
\begin{lem}\label{lem:Fourier_Coefficients}
For any $p\in [1,\infty]$, $F^{p}_{\red}(\alpha,u)$ is a reduced crossed product.
\end{lem}
\begin{proof} Assume that $p<\infty$.
We identify $F^{p}_{\red}(\alpha,u)$
with  $\overline{\Lambda_p(F(\alpha,u))}\subseteq B(\ell^p(X\times G))$ and use the notation from the proof of Lemma \ref{lem:regular_representation}. For each $(x,t)\in X\times G$ we let  $P_{x,t}(\xi):=\xi(x,t) \mathds{1}_{x,t}$ be a  contractive projection in $B(\ell^p(X\times G))$ onto the one-dimensional space $\C \mathds{1}_{x,t}$.  For any $b\in B(\ell^p(X\times G))$  the series
$$
\E_1(b):=\sum\limits_{(x,t)\in X\times G}P_{x,t} b  P_{x,t}
$$
 is strongly convergent and $\|\E_1(b)\|\leq \|b\|$, as for any $\xi\in \ell^p(X\times G)$ we get
$$
\|E_1(b)\xi\|^{p}=\sum\limits_{(x,t)\in X\times G} \|P_{x,t} b   \xi(x,t) \mathds{1}_{x,t}\|^p
\leq \|b\|^{p}  \sum\limits_{(x,t)\in X\times G} \|\xi(x,t) \mathds{1}_{x,t}\|^p=\|b\|^{p} \|\xi\|^{p}
$$
(these sums are in fact countable, as $\xi(x,t)\neq 0$ only for a countable set of pairs $(x,t)$). 
Also if $b\in F(\alpha,u)\subseteq F^{p}_{\red}(\alpha,u)\subseteq B(\ell^p(X\times G))$, then clearly 
$\E_1(b)=E_1(b)=b(1)$. 
Thus $\E_1$ is a contractive extension of $E_1$.
For any $t\in G$, the formula $\E_t(b):=\E_1(b v_{t}^{-1})$ defines a contractive map $\E_t:F^{p}_{\red}(\alpha,u)\to C_0(X)$ that extends $E_t$.
Hence $\{\E_t\}_{t\in G}$ is the Fourier decomposition of $F^{p}_{\red}(\alpha,u)$.
For any $s,t\in G$, $x\in X$ we have
\begin{equation}\label{eq:Fourier_coeffficient_map_reduced}
\E_s(b)\mathds{1}_{x,t}:=P_{x,t} b  \overline{\pi}(u(s,s^{-1}t)^{-1})   \mathds{1}_{\varphi_s^{-1}(x),s^{-1}t} .
\end{equation}
The operators  of the form $av_{s}$, $a\in C_0(X)$, $s\in G$, and hence also every $b\in  F^{p}_{\red}(\alpha,u)$,  maps  $\mathds{1}_{x,t}$ to the closed linear span of
elements $\mathds{1}_{\varphi_s(x),st}$, $s\in G$. Therefore if $b\neq 0$ there are $s,t\in G$ and $x\in X$
such that $P_{\varphi_s(x),st} b \mathds{1}_{x,t}\neq 0$. 
By \eqref{eq:Fourier_coeffficient_map_reduced}, the latter is equivalent to  $\E_s(b)\mathds{1}_{\varphi_s(x),st}\neq 0$.
Hence $b\neq 0$ implies $\E_{s}(b)\neq 0$ for some $s\in G$. Consequently, $\bigcap_{t\in G}\ker(\E_t)=\{0\}$.

For $p=\infty$ the proof above  works with $\ell^p(X\times G)$ replaced by $c_0(X\times G)$, but also the assertion follows from 
the case $p=1$ using the anti-isomorphism  $F^{\infty}_{\red} (\alpha,u)\stackrel{anti}{\cong}F^{1}_{\red} (\alpha,u)$.
\end{proof}
\begin{cor}\label{cor:Fourier_maps_on_full}
For any $p\in [1,\infty]$,  $F^{p}(\alpha,u)$ admits the Fourier decomposition 
$\{E_{p,t}\}_{t\in G}$ and $\bigcap_{t\in G}\ker(E_{p,t})=\ker\Lambda_p$
is the kernel of the regular representation.
\end{cor}
\begin{proof}
Put $E_{p,t}:=\E_t\circ \Lambda_p$,  $t\in G$, where  $\Lambda_p:F^{p}(\alpha,u)\to F^p_{\red}(\alpha,u)$ is the regular representation and $\{\E_{t}\}_{t\in G}$ is the Fourier decomposition of $F^p_{
\red}(\alpha,u)$ given by Lemma \ref{lem:Fourier_Coefficients}.
\end{proof}
There  are canonical ways of producing  (reduced) Banach algebra or $*$-Banach algebra crossed products  from 
 other (reduced) Banach algebra crossed products: 
\begin{lem}\label{lem:reduced_from_other_reduced}
 Let $\|\cdot \|_{\RR}$ and $\|\cdot\|_{\RR_{i}}$, $i\in I$, be norms on $F(\alpha,u)$ that  define Banach algebra  crossed products $F_{\RR}(\alpha,u)$ and $F_{\RR_i}(\alpha,u)$, $i\in I$, 
for $(\alpha,u)$.
Then the formulas 
$$
\|a\|_{\RR^*}:=\|a^*\|_{\RR}, \qquad \|a\|_{\{ \RR_i\}_{i\in I}}:=\sup_{i\in I}\|a\|_{R_i}, \qquad a\in F(\alpha,u),
$$
define norms that  yield Banach algebra  crossed products  $F_{\RR^*}(\alpha,u)$, $F_{\{ \RR_i\}_{i\in I}}(\alpha,u)$ for $(\alpha,u)$.  
In particular,  the norm $\|a\|_{\RR,*}:=\max\{\|a\|_{\RR},\|a^*\|_{\RR}\}$ defines a crossed product $F_{\RR,*}(\alpha,u)$ which is  a $*$-Banach algebra.
Moreover, $F_{\RR^*}(\alpha,u)$ and $F_{\RR,*}(\alpha,u)$ are reduced if $F_{\RR}(\alpha,u)$ is,
and $F_{\{ \RR_i\}_{i\in I}}(\alpha,u)$ is reduced if and only if  all $F_{\RR_i}(\alpha,u)$, $i\in I$, are reduced.
\end{lem}
\begin{proof}
Clearly, $ \|\cdot\|_{\RR^*}$ is a submultiplicative norm with   $\|b\|_{\RR^*}\leq \sum_{t\in G} \|b(t)\|_{\infty}$ for  $b\in F(\alpha,u)$ (because $\|\cdot\|_{\RR}$ has these properties). 
The involution on $F(\alpha,u)$ extends to an antimultiplicative antilinear isometry $*$ from 
$F_{\RR^*}(\alpha,u)$ onto $F_{\RR}(\alpha,u)$. In particular, if $F_{\RR}(\alpha,u)$ admits the FD, then so does $F_{\RR^*}(\alpha,u)$, as  then 
$
\|b(1)\|=\|b(1)^*\|\leq \|b^*\|_{\RR}= \|b\|_{\RR^*}, 
$ for all $b\in F(\alpha,u)$. Moreover, if  $\{E_{t}^{\RR}\}_{t\in G}$ is the FD for $F_{\RR}(\alpha,u)$ and $\{E_{t}^{\RR^*}\}_{t\in G}$ is the FD for $F_{\RR^*}(\alpha,u)$,
then by continuity  and the formula for involution on $F(\alpha,u)$ we get 
$$
E_t^{\RR^*}(a)=\alpha_{t}(E_{t^{-1}}^{\RR}(a^*))^* u(t,t^{-1})^*, \qquad t\in G, \, a\in F_{\RR^*}(\alpha,u).
$$
This readily implies that $F_{\RR^*}(\alpha,u)$ is reduced if and only if $F_{\RR}(\alpha,u)$ is.

Clearly, $\|\cdot \|_{\{ \RR_i\}_{i\in I}}$ is a submultiplicative norm with   $\|b\|_{\{ \RR_i\}_{i\in I}}\leq \sum_{t\in G} \|b(t)\|_{\infty}$ for  $b\in F(\alpha,u)$ (because each $\|\cdot\|_{\RR_{i}}$ has these properties).
If there is at least one $i_0\in I$ such that $F_{\RR_{i_0}}(\alpha,u)$ admits the FD, then $F_{\{ \RR_i\}_{i\in I}}(\alpha,u)$ admits the FD, as  then
$\|b(1)\|_{\infty}\leq \|b\|_{R_{i_0}}\leq  \|b\|_{\{ \RR_i\}_{i\in I}}$ for all $b\in F(\alpha,u)$. 
For each $i\in I$,  let $\{E_{t}^{\RR_i}\}_{t\in G}$  be the FD for $F_{\RR_{i}}(\alpha,u)$, 
and denote by  $\{\EE_{t}\}_{t\in G}$  the FD for $F_{\{ \RR_i\}_{i\in I}}(\alpha,u)$. Note that $\pi(a):=\prod_{i\in I} a$ is an isometric embedding of $F_{\{ \RR_i\}_{i\in I}}(\alpha,u)$
into the direct product $\prod_{i\in I} F_{\RR_{i}}(\alpha,u)$. Moreover, $\prod_{i\in I} \EE_{t}= \prod_{i\in I} E_{t}^{\RR_i}\circ  \pi$   (as maps from  $F_{\{ \RR_i\}_{i\in I}}(\alpha,u)$ to  $\prod_{i\in I} C_0(X)$).
Thus $\bigcap_{t\in G}\ker(\EE_t)=0$ if $\bigcap_{t\in G}\ker( E_{t}^{\RR_i})=0$ for all $i\in I$. That is, $F_{\{ \RR_i\}_{i\in I}}(\alpha,u)$ is reduced if all $F_{\RR_{i}}(\alpha,u)$, $i\in I$, are reduced.
\end{proof}
\begin{ex}
\label{ex:Banach_*-algebras}
Let $p,q\in[1,\infty]$ satisfy $1/p+1/q=1$, and let $\|\cdot\|_{L^p}$ and $\|\cdot\|_{L^p,\red}$ be norms in $F^p(\alpha,u)$ and $F^p_{\red}(\alpha,u)$, respectively.
By Remark \ref{rem:consequences_of_the_groupoid_paper}, we have  $\|a^*\|_{L^p}=\|a\|_{L^q}$,  $\|a^*\|_{L^p,\red}=\|a\|_{L^q,\red}$ for $a\in F(\alpha,u)$. Thus in the notation of Lemma \ref{lem:reduced_from_other_reduced},  we have 
 $\|a\|_{L^p,*}=\max\{\|a\|_{L^p},\|a\|_{L^q}\}$ and $\|a\|_{L^p,\red,*}=\max\{\|a\|_{L^p,\red},\|a\|_{L^q,\red}\}$. Hence 
$$
F^{p,*}(\alpha,u):=\overline{F(\alpha,u)}^{\|\cdot\|_{L^p,*}}\qquad\text{and}\qquad F^{p,*}_{\red}(\alpha,u):=\overline{F(\alpha,u)}^{\|\cdot\|_{L^p,\red,*}}
$$
are Banach $*$-algebra crossed products for $(\alpha,u)$. The algebra $F^{p,*}(\alpha,u)$ is universal for covariant representions of $(\alpha,u)$ on $\ell^{\infty}$-direct sums of 
$L^p$ and $L^q$-spaces, while $F^{p,*}_{\red}(\alpha,u)$ is its reduced version. 
The reduced group Banach algebra $B^{p, *}_{\mathrm{r}} (G)$ introduced in \cite{LiYu}, see \cite[Definition 1.6]{Phillips19}, 
is the special case of $F^{p,*}_{\red}(\alpha,u)$ where $X=\{x\}$ and $u\equiv 1$. 
In the  context of twisted groupoids, the algebras $F^{p,*}_{\red}(\alpha,u)$ were studied in  \cite{Austad_Ortega},  see \cite[Definition 3.7]{Austad_Ortega}.
In particular, we have 
$
F^{1,*}(\alpha,u)=F^{1,*}_{\red}(\alpha,u)=F^{\infty,*}(\alpha,u)=F^{\infty,*}_{\red}(\alpha,u)
$
and this $*$-Banach algebra is equipped with the so called Hahn's \emph{$I$-norm} $\|\cdot\|_{I}$ considered by a number of authors (cf.
\cite{BKM1}, \cite{Austad_Ortega}, and the sources cited there). We denote this $*$-Banach algebra by $F_{I}(\alpha,u)$ and call it \emph{Hahn's crossed product}.
 \end{ex}
To cover all these examples under one umbrella we will associate crossed products to any set of parameters $P\subseteq [1,\infty]$. 
\begin{defn}\label{def:L^P_crossed_products}
For any non-empty $P\subseteq [1,\infty]$ we put
$$
F^P(\alpha,u):=\overline{F(\alpha,u)}^{\|\cdot\|_{P}}\quad \textrm{and} \qquad F^P_{\red}(\alpha,u):=\overline{F(\alpha,u)}^{\|\cdot\|_{P,\red}},
$$
 where $\|b\|_{P}:=\sup_{p\in P}\|b\|_{L^p}$ and  $\|b\|_{P,\red}:=\sup_{p\in P}\|b\|_{L^p,\red}$, $b\in F(\alpha,u)$.
We denote by $\Lambda_{P}:F^P(\alpha,u)\to F^P_{\red}(\alpha,u)$  the canonical representation (which is the identity on $F(\alpha,u)$). 
For convenience we also put 
\begin{equation}\label{eq:emptyparameters_definition}
F^\emptyset(\alpha,u):=F^\emptyset_{\red}(\alpha,u):=F(\alpha,u)\quad   \text{  and  }\quad \Lambda_{\emptyset}:=\text{id}|_{F(\alpha,u)}.
\end{equation}

\end{defn}
\begin{rem}\label{rem:properties_of_F^P}
For $\emptyset \neq P\subseteq [1,\infty]$, the algebra $F^P(\alpha,u)$ is universal for representations of $F(\alpha,u)$ on ($\ell^{\infty}$-direct sums of) $L^p$-spaces  for all $p\in P$. 
By Lemma \ref{lem:reduced_from_other_reduced} and Corollary \ref{cor:Fourier_maps_on_full}, $F^P(\alpha,u)$ admits  the Fourier decomposition 
$\{E_{P,t}\}_{t\in G}$ and $\bigcap_{t\in G}\ker(E_{P,t})=\ker\Lambda_P$.
Similarly,  $F^P_{\red}(\alpha,u)$ is a reduced crossed product that can be  isometrically represented on the $\ell^{\infty}$-direct sum of spaces $\ell^p(X\times G)$, $p\in P$.
When $P=\{p\}$ is a singleton we recover the $L^p$-operator algebra crossed products from Definition \ref{def:L^p_crossed_products}. When $P=\{p,q\}$ with $1/p+1/q=1$ we get the Banach $*$-algebra crossed products
from Example \ref{ex:Banach_*-algebras}. In general, $F^P(\alpha,u)$ is a Banach $*$-algebra whenever the set $1/P:=\{1/p:p\in P\}$ is symmetric with respect to $1/2$. If $\{1,\infty\}\subseteq P$, then $F^P(\alpha,u)=F^P_{\red}(\alpha,u)=F_{I}(\alpha,u)$ is
the Hahn's crossed product, by Remark \ref{rem:consequences_of_the_groupoid_paper}.
\end{rem}
\begin{cor}\label{cor:properties_of_F^P}
 For any non-empty $P\subseteq [1,\infty]$ and any $b\in F(\alpha,u)$ we have
$$
\|b\|_{P,\red}=\sup\{\|\Ind_p(\pi)(b)\|: \pi:C_0(X)\to B(E) \text{ is a non-degenerete representation } 
$$
$$
   \qquad \qquad   \qquad  \text{ where $E$ is an $L^p$-space for $p\in P\setminus \{\infty\}$ or a $C_0$-space, if $\infty\in P$}\}.
$$
If $P\subseteq \{1,\infty\}$, then $F^P(\alpha,u)=F^P_{\red}(\alpha,u)$. If the action $\varphi$ is amenable, then $F^P(\alpha,u)=F^P_{\red}(\alpha,u)$ for any $P\subseteq [1,\infty]$.
\end{cor}
\begin{proof} The formula for the norm follows from Proposition \ref{prop:reduced_norm_is_induced}. 
The statement when $P \subseteq \{1,\infty\}$ holds by Remark \ref{rem:consequences_of_the_groupoid_paper} and Definition \ref{eq:emptyparameters_definition}.
When $\varphi$ is amenable, then full equals to reduced by Theorem \ref{thm:ambenability_reduced_equals_full}.
\end{proof}
\section{Topological freeness and intersection properties}
\label{sec:the_main_result}

The term topological freeness was probably coined by Tomiyama. However, the condition appeared much earlier in the work of Zeller-Meier where it was used to characterise when the coefficient algebra is maximal abelian in the  reduced $C^*$-algebra crossed product. We generalise this to Banach algebras.
\begin{defn}[cf. {\cite[Definition 2.1(c)]{Tomiyama}}]
We call $\varphi$ \emph{topologically free} if the sets  $\{ x\in X: \varphi_t(x)=x\}$, $t\in G\setminus\{1\}$, have empty interiors.
\end{defn}

\begin{prop}[cf. {\cite[Proposition 4.14]{Zeller-Meier}}]
Let $F_{\RR}(\alpha,u)$ be a reduced crossed product. The algebra $C_0(X)$ is a maximal commutative subalgebra of $F_{\RR}(\alpha,u)$ if and only if $\varphi$ is topologically free.
\end{prop}
\begin{proof}
If $\varphi$ is not topologically free, then there is $t\in G\setminus \{1\}$ and a non-zero $a\in C_0(X)$ with
$\supp(a):=\overline{\{x:a(x)\neq 0\}}\subseteq \{x\in X:\varphi_t(x)=x\}$. Then  $b:=av_t\notin C_0(X)$ and $b$ commutes with every $f\in C_0(X)$ because 
$ b\cdot  f= av_t\cdot f= a\alpha_t(f) v_{t}=f a v_t= f \cdot b$. So $C_0(X)$ is not maximal abelian subalgebra of $F_{\RR}(\alpha,u)$ (here $F_{\RR}(\alpha,u)$ does  not need to be reduced).
Conversely, assume that  $C_0(X)$ is not a maximal abelian subalgebra of $F_{\RR}(\alpha,u)$, and let $b\in F_{\RR}(\alpha,u) \setminus C_0(X)$ commute with all elements in $C_0(X)$. 
Since $b\not\in C_0(X)$ and $F_{\RR}(\alpha,u)$ is reduced, there is $t\in G\setminus \{1\}$ such that $E_t(b)\neq 0$. For any $f\in C_0(X)$, using \eqref{eq:twisted_module_maps} 
and that $b$ commutes with $f$ we get $f E_t(b)= E_t(fb)= E_t(bf)=E_t(b)\alpha_t(f)$, which implies that $\supp(E_t(b))\subseteq  \{x\in X:\varphi_t(x)=x\}$.
Hence $\varphi$ is not topologically free.
\end{proof}

Topological freeness implies what in the context of $C^*$-algebras  R\o rdam  calls  the \emph{pinching property}, see \cite[Definition 3.13]{Rordam}. 
We use it to show that any Banach algebra completion of $F(\alpha,u)$  (in fact any Hausdorff completion into which $C_0(X)$ embeds)  admits the Fourier decomposition.

\begin{lem}[pinching property]\label{lem:pinching_property}
If $\varphi$ is topologically free, then for every $b\in F(\alpha,u)$, $t\in G$, and $\varepsilon >0$ 
there exists $h\in C_c(X)^+$ of norm one such that 
$$
\| E_t(b)\|\leq\| hE_t(b)h\|+\varepsilon, \quad \|hE_t(b)h-hb (h\delta_t)^*\|\leq\varepsilon.
$$
\end{lem}
\begin{proof}
For $t=1$ this follows by (the proof of)  \cite[Proposition 2.4]{Exel_Laca_Quigg}. For general $t\in G$ we may use that $E_t(b)=\lim_{i} E_1(b \cdot  (\mu_i \delta_{t})^*)$ 
for  any contractive approximate unit $(\mu_i)$ in $C_0(X)$. Namely, putting $b_i:=b \cdot  (\mu_i \delta_{t})^*$ 
 from some point on we have $\|E_t(b)-E_1(b_i)\|\leq \varepsilon/3$, and by the previous remark we may find a norm one $h\in C_c(X)^+$ such that 
$
\|E_1(b_i)\|\leq \| hE_1(b_i)h\|+\varepsilon/3$ and $\|hE_1(b_i)h-hb_ih\|\leq\varepsilon/3.
$
Using these three inequalites we get
$$\| E_t(b)\|\leq \varepsilon/3+ \|E_1(b_i)\|\leq \| hE_1(b_i)h\|+2\varepsilon /3 \leq \| hE_t(b)h\|+\varepsilon, $$
$$
\|hE_t(b)h-h b_ih\|\leq\varepsilon/3 +\|hE_1(b_i)h-hb_ih\|\leq 2\varepsilon/3.
$$
Since  $\lim_{i}b_ih =b (h\delta_t)^*$ we get the assertion.
\end{proof}
\begin{lem}[{\cite[Theorem 10]{Bonsall}}]\label{lem:incompressibility}
If $A$ is a $C^*$-algebra, so in particular, if $A=C_0(X)$, then any injective representation $\psi: A\to B$ into a Banach algebra $B$ is isometric.
\end{lem}
\begin{prop}\label{prop:Fourier_decomposition}
Assume that $\varphi$ is topologically free and let $F_{\RR}(\alpha,u)$ be a crossed product.
For any representation $\psi:F_{\RR}(\alpha,u)\to B$  which is injective on $C_0(X)$ and has dense range,   
there are contractive linear maps $E^{\psi}_{t}:B\to C_0(X)$, $t\in G$,
such that $E^{\psi}_{t}(\psi(b))=E_t^{\RR}(b)$ for  $b\in F_{\RR}(\alpha,u)$, $t\in G$. In particular, $\ker\psi\subseteq \bigcap_{t\in G}\ker(E_t^{\RR})$.
\end{prop}
\begin{proof} Fix $t\in G$. Take any $b\in F(\alpha,u)\subseteq F_{\RR}(\alpha,u)$ and $\varepsilon>0$.
Let  $h\in C_c(X)^+$ be as in Lemma \ref{lem:pinching_property}. Using  that $\psi$ is isometric on $C_0(X)$, by Lemma \ref{lem:incompressibility}, we get
\begin{align*}
\|E_t^{\RR}(b)\| &= \|E_t(b)\|\leq \| h E_t(b)h\|+\varepsilon=\|\psi(h E_t(b)h)\|+\varepsilon 
\\
&\leq\|\psi(h b (h\delta_t)^*)\|+2\varepsilon \leq\|\psi(b)\|+2\varepsilon.
\end{align*}
This  shows that $\|E_t^{\RR}(b)\|\leq  \|\psi(b)\|$. Hence  the formula $E^{\psi}_{t}(\psi(b)):=E_t^{\RR}(b)$, $b\in F(\alpha,u)$, yields a well defined linear contraction 
$\psi(F(\alpha,u))\to C_0(X)$ which extends to the contraction $E^{\psi}_{t}:B\to C_0(X)$ satisfying $E^{\psi}_{t}(\psi(b))=E_t^{\RR}(b)$ for
 $b\in F_{\RR}(\alpha,u)$. 
\end{proof}
The relationship between topological freeness and  ideals in $C^*$-crossed products was probably first made explicit  by O'Donovan, for $\Z$-actions, see  \cite[Theorem 1.2.1]{OD}.
Sierakowski in \cite{Sierakowski} called the relevant condition the intersection property. 
\begin{defn}[cf. {\cite[Definitions 5.5, 5.6]{Kwa-Meyer}}]
  \label{def:visible}
	
  An inclusion of Banach algebras $A\subseteq B$ has the \emph{generalised     intersection property} if there is the largest  ideal
  $\Null$ in $B$  with $\Null\cap A=\{0\}$.  Then we
  call~\(\Null\) the \emph{hidden ideal} for $A\subseteq B$. If $\Null=\{0\}$, that is if for every non-zero ideal $J$ in $B$ we have $J\cap A \neq \{0\}$ 
	we say that $A\subseteq B$ has the \emph{intersection property} or that the subalgebra $A$ \emph{detects ideals} in the algebra $B$.
		\end{defn}
		\begin{rem}\label{rem:reduced_is_necessary} By the last part of Lemma \ref{lem:Fourier_decomposition_maps}, if $F_{\RR}(\alpha,u)$ is a Banach algebra crossed product and $C_0(X)$ detects ideals in $F_{\RR}(\alpha,u)$, then $F_{\RR}(\alpha,u)$
is necessarily reduced. 
\end{rem}
	Clearly	$A\subseteq B$  has the generalised intersection property with hidden ideal $\Null$  if and only if any representation 
$\pi:B\to C$ injective on $A$ descends to a (necessarily injective) representation $B/\Null\to C$. Thus Proposition \ref{prop:Fourier_decomposition} immediately gives

\begin{cor}\label{cor:general_generalised_intersection}
If $\varphi$ is topologically free, then for any crossed product $F_{\RR}(\alpha,u)$
the inclusion $C_0(X)\subseteq F_{\RR}(\alpha,u)$ has the generalised intersection property with the hidden ideal $\bigcap_{t\in G}\ker(E_t^{\RR})$; in particular, 
$C_0(X)$ detects ideals in all reduced crossed products.
\end{cor}
When applied to   untwisted crossed products that have `$C_0(X)$-trivial representations', the implication in the above corollary  can be reversed. For $C^*$-algebraic crossed products this was proved 
in \cite[Theorem 4.1]{Kawa-Tomi}, \cite[Theorem 2]{Arch_Spiel}. We now generalise these results to crossed products $F^P(\alpha,u)$ where $P\subseteq [1,\infty]$.
If $P$ is non-empty,  we define the \emph{$C_0(X)$-trivial representation}  $\Lambda_{P}^{\tr}$  of $F^P (\alpha)$ as an extension of the $\ell^{\infty}$-direct sum  $\oplus_{p\in P}\Lambda_p^{\tr}$  of   representations from Example \ref{ex:trivial_rep}. We also put $\Lambda_{\emptyset}^{\tr}:=\Lambda_{[1,\infty]}^{\tr}|_{F (\alpha)}$ and recall that $F^\emptyset (\alpha)=F (\alpha)$.
\begin{thm}\label{thm:generalised_intersection_property}
Let $\alpha$ be an  action of a discrete group $G$ on $C_0(X)$. For any  $P\subseteq [1,\infty]$,  the following conditions are equivalent:
\begin{enumerate}
\item\label{enu:generalised_intersection_property1} The dual action $\varphi$ of $G$ on $X$ is topologically free.
\item\label{enu:generalised_intersection_property2} $C_0(X)\subseteq  F^P(\alpha)$  has the generalised intersection property with the hidden ideal being the kernel $\ker\Lambda_P$
of the regular representation $\Lambda_P:F^P (\alpha)\to F_{\red}^P (\alpha)$.
\item\label{enu:generalised_intersection_property3} We have  $\ker(\Lambda_P^{\tr})\subseteq\ker(\Lambda_P)$ where $\Lambda_{P}^{\tr}$ is the $C_0(X)$-trivial representation of  $F^P (\alpha)$.
\item\label{enu:generalised_intersection_property4} $C_0(X)$ detects ideals in $F^T(\alpha)$ for some (and hence all)  $T\subseteq \{1,\infty\}$.
\end{enumerate}
\end{thm}
\begin{proof}
\ref{enu:generalised_intersection_property1} implies \ref{enu:generalised_intersection_property2} by Corollary \ref{cor:general_generalised_intersection} and 
Remark \ref{rem:properties_of_F^P}.
 \ref{enu:generalised_intersection_property2} implies \ref{enu:generalised_intersection_property3} because $\ker(\Lambda_P^{\tr})\cap C_0(X)=\{0\}$.
To show that \ref{enu:generalised_intersection_property3} implies \ref{enu:generalised_intersection_property1}  assume that $\varphi$ is not topologically free. Then 
there are $t\in G\setminus\{1\}$ and an open non-empty set $U\subseteq X$ such that $\varphi_t|_{U}=\text{id}_{U}$. Take any non-zero $a\in C_0(U)\subseteq C_0(X)$. 
Then  $b:=a-\delta_{t^{-1}}a\in F(\alpha)$ is non-zero, and so $\Lambda_P(b)\neq 0$,  but   \eqref{eq:trivial_rep} readily implies that $ \Lambda_P^{\tr}(b)=0$. 
Hence $\ker(\Lambda_P^{\tr})\not\subseteq\ker(\Lambda_P)$. Thus \ref{enu:generalised_intersection_property1}-\ref{enu:generalised_intersection_property3} are equivalent. These conditions are independent of $P$ because \ref{enu:generalised_intersection_property1} is. 
Condition \ref{enu:generalised_intersection_property4} is equivalent to \ref{enu:generalised_intersection_property2} for $P=T$,
because $F^T(\alpha)=F^T_{\red}(\alpha)$ by Corollary   \ref{cor:properties_of_F^P}. Hence  all conditions \ref{enu:generalised_intersection_property1}-\ref{enu:generalised_intersection_property4} are equivalent.
\end{proof}
\begin{cor}\label{cor:intersection_property_full_untwisted}
 For any $P\subseteq [1,\infty]$,  $C_0(X)$ detects ideals in $F^P(\alpha)$
if and only if the regular representation $\Lambda_P$ is injective and  $\varphi$ is topologically free.
 \end{cor}
\begin{proof}
Since $\ker(\Lambda_P)\cap C_0(X)=\{0\}$, we see that $C_0(X)$ detects ideals in $F^P(\alpha)$ if and only if $\ker(\Lambda_P)=\{0\}$ and condition \ref{enu:generalised_intersection_property2} in Theorem \ref{thm:generalised_intersection_property} holds. 
\end{proof}

By Corollary \ref{cor:general_generalised_intersection}, topological freeness implies 
that $C_0(X)$ detects ideals in every reduced algebra $F^p_{\red}(\alpha)$, $p\in (1,\infty)$, but in general (when the action is not amenable)
 the converse is not true. 
For instance,   the main result of \cite{PH} implies that if $G$ is a group with the Powers property (so for instance, a free group $G=\F_n$ with $n>1$ generators), then for any minimal action $\varphi$    
on a compact  space $X$  the crossed products $F^p_{\red}(\alpha)$, $p\in (1,\infty)$, are simple, and so $C_0(X)$ detects ideals in $F^p_{\red}(\alpha)$ but the action might be  trivial. This in particular applies to actions on $X=\{x\}$. Then the resulting crossed products are \emph{group Banach algebras} that we denote by $F(G,u)$,
$F_{\red}^p(G,u)$ and so on. For  group  algebras Phillips  extended in \cite{Phillips19} the result of \cite{PH}  to all $C^*$-simple groups.  So for `large groups'
topological freeness is not necessary for simplicity of $F^p_{\red}(\alpha)$.
Adding a twist to the picture makes things even more interesting, as then  topological freeness is not necessary even  for  simplicity of algebras  of   finite abelian groups:
\begin{ex}
Any matrix algebra $M_{n}(\C)$ is a twisted group algebra where  $G:=\Z_n\times\Z_n$ and 
$u( (p ,r), (s,t)):=e^{-2\pi i \frac{rs}{n}}$, $p,r,s,t\in \Z_n$. The identification (that disregards norms)  goes
via the map $F(G,u)\ni a\mapsto \overline{a}\in M_{n}(\C) $ where $(\overline{a}x)(r):=\sum_{s,t\in \Z_n} a(s,t)e^{2\pi i\frac{rs}{n}} x(r-t)$ for $x=(x(1),...,x(n))\in \C^n$ and $r=1,...,n$.
Note that functions  supported on the subgroup  $H:=\Z_n\times \{0\}$ 
form a  subalgebra  of diagonal matrices in $F(G,u)=M_{n}(\C)$ that can be identified with $F(H)$.
In particular, $\C\cdot 1$ detects ideals in $F(G,u)$, because $F(G,u)$ is simple, but $\C\cdot 1$ does  not detect ideals in  $F(H) \subseteq F(G,u)$, unless $n=1$.
 \end{ex}
We now show that the only reason why topological freeness fails in the above examples is that  it  passes to subgroups while  detection of ideals does not pass to subalgebras (not even those coming from subgroups). Namely, let  $H\subseteq G$ be a subgroup of $G$. Topological freeness of $\varphi:G\to \Homeo(X)$ implies topological freeness of the restricted action $\varphi|_{H}:H\to \Homeo(X)$.
Also any twisted action $(\alpha,u)$ of $G$ restricts to a twisted action $(\alpha|_H,u|_H)$ of $H$ (formally we should write $u|_{H\times H}$).
Clearly, $F(\alpha|_H,u|_H)\subseteq F(\alpha,u)$. If $\RR$ is a class of representations of $F(\alpha,u)$, we denote by $\RR|_H$ their restrictions to $F(\alpha|_H,u|_H)$.
Then  $F_{\RR|_{H}}(\alpha|_{H},u|_{H})\subseteq  F_{\RR}(\alpha,u)$ is a Banach subalgebra, and $F_{\RR|_{H}}(\alpha|_{H},u|_{H})$ admits the FD or  is reduced  if $F_{\RR}(\alpha,u)$
has this property. This applies to full  crossed products in Definition \ref{def:L^P_crossed_products} giving $F^P(\alpha|_{H},u|_{H})\subseteq  F^P(\alpha,u)$
for any $P\subseteq [1,\infty]$. For reduced crossed products this is not clear.
We certainly have $F^P_{\red}(\alpha|_{H},u|_{H})\subseteq  F^P_{\red}(\alpha,u)$ when  $ P\subseteq \{1,2,\infty\}$ (as for $p=1,\infty$ the  reduced and full crossed products coincide, and for $p=2$ the inclusion is 
 a well known $C^*$-algebraic fact).
We also have $F^P_{\red}(\alpha|_{H},u|_{H})\subseteq  F^P_{\red}(\alpha,u)$  whenever $F^P_{\red}(\alpha|_{H},u|_{H})=F^P(\alpha|_{H},u|_{H})$, which  holds for instance  when the restricted action $\varphi|_{H}$
is amenable.

The next result  is inspired by the corresponding result in \cite{Kwa-Meyer2} proved for  $C^*$-algebras associated to Fell bundles over inverse semigroups, cf. \cite[Proposition 6.4]{Kwa-Meyer2}.
\begin{thm}\label{thm:twisted_generalised_intersection_property}
Let $(\alpha,u)$ be a twisted action of a discrete group $G$ on $C_0(X)$. For any  $P\subseteq [1,\infty]$,  the following conditions are equivalent:
\begin{enumerate}
\item\label{enu:twisted_generalised_intersection_property1} The dual action $\varphi:G\to \Homeo(X)$ is topologically free.
\item\label{enu:twisted_generalised_intersection_property2} For any subgroup $H\subseteq G$,  $C_0(X)\subseteq F^P (\alpha|_{H},u|_{H})$ has the generalised intersection property with
hidden ideal being the kernel of  $\Lambda_P:F^P (\alpha|_{H},u|_{H})\to F_{\red}^P (\alpha|_{H},u|_{H})$.
\item\label{enu:twisted_generalised_intersection_property3}  $C_0(X)$ detects ideals in all $F_{\red}^P (\alpha|_{H},u|_{H})$ for any subgroup $H\subseteq G$.
\item\label{enu:twisted_generalised_intersection_property4}  $C_0(X)$ detects ideals in all $F_{\red}^P (\alpha|_{H},u|_{H})$ for any cyclic subgroup $H\subseteq G$.
\end{enumerate}
\end{thm}
\begin{cor}\label{cor:intersection_property_full_twisted}
Assume the action $\varphi$ is amenable and let $u\in Z^2(G, C_u(X))$ and $P\subseteq  [1,\infty]$.
Then $\varphi$ is topologically free if and only if   $C_0(X)$ detects ideals in all intermediate subalgebras 
$F^P (\alpha|_{H},u|_{H})\subseteq F^P (\alpha,u)$, where  $H\subseteq G$ is a subgroup. 
 \end{cor}
\begin{proof}
Amenability of $\varphi:G\to \Homeo(X)$ passes  to the restricted action $\varphi:H\to \Homeo(X)$.
Hence $F^P (\alpha|_{H},u|_{H})=F_{\red}^P (\alpha|_{H},u|_{H})$ by Corollary \ref{cor:properties_of_F^P}.
In particular, $F^P (\alpha|_{H},u|_{H})$ embeds into $F^P (\alpha,u)$.
\end{proof}

For the proof of Theorem \ref{thm:twisted_generalised_intersection_property} we need some preparations. 
A set  $U\subseteq X$  is \emph{$\varphi$-invariant} if  $\varphi_{t}(U)=U$ for all $t\in G$.  For any such set the twisted action $(\alpha,u)$ of $G$ restricts to a twisted action $(\alpha^U,u^U)$ of $G$ where  
the action  dual  to $\alpha^U$ is given by homeomorphisms  $\varphi^{U}_t:=\varphi_{t}|_{U}$, $t\in G$, and  we put $u^U(s,t):=u(s,t)|_{U}$ for $s,t\in G$. 
Assume $U$ is $\varphi$-invariant and open. 
Then $F(\alpha^U,u^U)=\ell^1(G,C_0(U))\subseteq  \ell^1(G,C_0(X))=F(\alpha,u)$ is naturally  a $^*$-Banach subalgebra, and in fact an ideal in $F(\alpha,u)$. If $\RR$ is a class of representations of $F(\alpha,u)$, we denote by $\RR^U$ their restrictions to $F(\alpha^U,u^U)$.
Then again, $F_{\RR^U}(\alpha^U,u^U)\subseteq  F_{\RR}(\alpha,u)$ is an ideal, and $F(\alpha^U,u^U)$ admits the FD or  is reduced  if $F_{\RR}(\alpha,u)$
has this property. Modulo compression of the space, this applies to reduced $L^p$-operator algebra crossed products giving $F^P_{\red}(\alpha^U,u^U)\subseteq  F^P_{\red}(\alpha,u)$
for any $P\subseteq [1,\infty]$. For full crossed products this is not clear.
We  have $F^P(\alpha^U,u^U)\subseteq  F^P(\alpha,u)$ for $P\subseteq \{1,2,\infty\}$, as for $p=1,\infty$ the reduced and full crossed products coincide and for $p=2$ it is
 a  known $C^*$-algebraic fact, see \cite[Proposition 21.15]{Exel_book}.
For other $P$ we  certainly have this whenever $F^P(\alpha^U,u^U)=F^P_{\red}(\alpha^U,u^U)$ which holds for instance when the restricted action $\varphi^{U}:G\to \Homeo(U)$
is amenable.

\begin{lem}\label{lem:restriction_of_action_to_set}
Assume $F_{\RR}(\alpha,u)$ is a Banach algebra crossed product. 
An ideal in $F_{\RR}(\alpha,u)$ is generated by its intersection with $C_0(X)$ if and only if it is of the form $F_{\RR^U}(\alpha^U,u^U)$ 
for some open $\varphi$-invariant $U\subseteq X$, and then $F_{\RR^U}(\alpha^U,u^U)\cap C_0(X)=C_0(U)$. If $C_0(X)$ detects ideals in $F_{\RR}(\alpha,u)$, then $C_0(U)$
detects ideals in $F_{\RR^U}(\alpha^U,u^U)$ for any open $\varphi$-invariant $U\subseteq X$.
\end{lem}
\begin{proof} 
If $J$ is an ideal in $F_{\RR}(\alpha,u)$, then $J\cap C_0(X)$ is an ideal in $C_0(U)$ and hence $J\cap C_0(X)=C_0(U)$ for an open set $U\subseteq X$.
Disintegrating the identity map on  $F_{\RR}(\alpha,u)$,  see Lemma \ref{lem:integration_disintegration}, and using  covariance relations, for any $a\in C_0(U)$ and $t\in G$ we
get  $\alpha_t(a)=v_t a v_{t}^{-1}\in J\cap C_0(X)=C_0(U)$. Hence $U$ is $\varphi$-invariant. 
Now let $U$ be any $\varphi$-invariant open set. The ideal generated by $C_0(U)$ in $F(\alpha,u)= \ell^1(G,C_0(X))$ is $F(\alpha^U,u^U)=\ell^1(G,C_0(U))$,
and hence the ideal generated by $C_0(U)$ in $F_{\RR}(\alpha,u)$ is $F_{\RR^U}(\alpha^U,u^U)$. 
Moreover, $F_{\RR^U}(\alpha^U,u^U)\cap C_0(X)=C_0(U)$.  Indeed, the  inclusion $C_0(U)\subseteq F_{\RR^U}(\alpha^U,u^U)\cap C_0(X)$ is obvious and for the opposite inclusion
note  that $E_1^{\RR}$ is the identity on $F_{\RR^U}(\alpha^U,u^U)\cap C_0(X)$,  and   $E_1^{\RR}(F_{\RR^U}(\alpha^U,u^U))=C_0(U)$ because 
$F_{\RR^U}(\alpha^U,u^U)$ is the closure of $F(\alpha^U,u^U)$ and $E_1(F(\alpha^U,u^U))=C_0(U)$.
This proves the first part of the assertion. For the second part, assume that  $C_0(U)$ does not detect ideals in $F_{\RR^U}(\alpha^U,u^U)$, so that there is a non-zero ideal $I$ in  $F_{\RR^U}(\alpha^U,u^U)$ with $I\cap C_0(U)=\{0\}$. 
Since $F_{\RR^U}(\alpha^U,u^U)$  has an approximate unit which is an approximate unit in $C_0(U)$, and $F_{\RR^U}(\alpha^U,u^U)$ is an ideal in 
$F_{\RR}(\alpha,u)$, we conclude that $I$ is an ideal in $F_{\RR}(\alpha,u)$ and $I\cap C_0(X)=I\cap C_0(U)=\{0\}$. Hence $C_0(X)$ does not detect ideals in $F_{\RR}(\alpha,u)$.
\end{proof}
\begin{lem}\label{lem:trivial_action_bundle_description}
Suppose that  $\varphi$  acts by identities, that is $\varphi_g=\text{id}_X$ for $g\in G$. 
For any non-empty $P\subseteq [1,\infty]$,  $F^P(\alpha,u)$ is isometrically isomorphic to the algebra of $C_0$-sections of 
a bundle $\B=\{ F_{\red}^P(G,u_x)\}_{x\in X}$ of twisted group reduced Banach algebras, where $u_x(s,t):=u(s,t)(x)$ and the isomorphism sends $b\mapsto \widehat{b}$ where $\widehat{b}(x)(t):=b(x,t)$, $x\in X$, $t\in G$.
\end{lem}
\begin{proof}
 Since $\alpha_{t}=\text{id}$ for all $t\in G$,  each $u_x \in Z^2(G,\mathbb{T})$, $x\in X$, is a 2-cocycle for $G$.  
Since $\varphi$  is trivial we have $\|\cdot\|_{L^1,\red}=\|\cdot\|_{L^{\infty},\red}$ by \eqref{eq:norms_for_1_and_infty}.
Hence we may assume that $P\subseteq [1,\infty)$.
A simple calculation shows that for any $b\in F(\alpha,u)$ and $p\in [1,\infty)$ we have
$$
\|b\|_{L^p,\red}=\|\Lambda_{p}(b)\|=\max_{x\in X}\left( \sum_{t\in G}|b(t)(x)|^{p}\right)^{\frac{1}{p}}.
$$
This calculation also shows that $\|\widehat{b}(x)\|_{F_{\red}^p(G,u_x)}=\left( \sum_{t\in G}|b(t)(x)|^{p}\right)^{\frac{1}{p}}$, for each $x\in X$. This implies that the map
$b\mapsto \widehat{b}$ is a well defined isometry from $F(\alpha,u)\subseteq F^P(\alpha,u)$ into the space of bounded sections of  $\B=\{ F_{\red}^P(G,u_x)\}_{x\in X}$ equipped with the supremum norm, as we have
$$
\|b\|=\sup_{s\in P}\max_{x\in X}\left( \sum_{t\in G}|b(t)(x)|^{p}\right)^{\frac{1}{p}}= \sup_{x\in X}\sup_{s\in P} \left( \sum_{t\in G}|b(t)(x)|^{p}\right)^{\frac{1}{p}}=\|\widehat{b}\|.
$$
Clearly, $b\mapsto \widehat{b}$  is an algebra homomorphism.
By Fell's reconstruction theorem \cite[II.13.18]{FD} there is a unique topology on $\B$ making sections $\widehat{b}$, $b\in F(\alpha,u)$, continuous.
Since elements of $C_c(G,C_c(X))\subseteq F(\alpha,u)$ are mapped into compactly supported sections of $\B$, and any compactly supported section can be approximated by
such elements, we conclude that $b\mapsto \widehat{b}$ extends to the isometric isomorphism $F^P(\alpha,u)\cong C_0(\B)$.
\end{proof}
\begin{proof}[Proof of Theorem \ref{thm:twisted_generalised_intersection_property}]
By Corollary \ref{cor:general_generalised_intersection}, and as topological freeness passes to subgroups, \ref{enu:twisted_generalised_intersection_property1} implies each of the conditions 
\ref{enu:twisted_generalised_intersection_property2}-\ref{enu:twisted_generalised_intersection_property4}, see also Corollary \ref{cor:Fourier_maps_on_full}. Implication 
\ref{enu:twisted_generalised_intersection_property3}$\Rightarrow$\ref{enu:twisted_generalised_intersection_property4}  is trivial, and  \ref{enu:twisted_generalised_intersection_property2} implies \ref{enu:twisted_generalised_intersection_property4} because cyclic groups are amenable and 
then the full and reduced crossed products coincide. 
Thus it suffices to prove 
\ref{enu:twisted_generalised_intersection_property4}$\Rightarrow$\ref{enu:twisted_generalised_intersection_property1}.
To this end,  assume $\varphi$ is not topologically free, so that
there are $t\in G\setminus\{1\}$ and an open non-empty set $U\subseteq X$ such that $\varphi_t|_{U}=\text{id}_{U}$. 
Let $H$ be the cyclic group generated by $t$. Note that $F_{\red}^P (\alpha|_{H},u|_{H})=F^P(\alpha|_{H},u|_{H})$  because $H$ is amenable. If $H\cong \Z$, 
then $F(\alpha|_{H},u|_{H})\cong F(\alpha|_{H})$, and therefore also $F^P(\alpha|_{H},u|_{H})\cong F^P(\alpha|_{H})$, by  Corollary \ref{cor:untwist_the_twisted}.
Thus in this case $C_0(X)$ does not detect ideals in   $F_{\red}^P (\alpha|_{H},u|_{H})=F^P(\alpha|_{H},u|_{H})$   by Theorem \ref{thm:generalised_intersection_property}.
Hence we may assume that $|H|<\infty$. By passing to a power of $t$ we may assume that $|H|$ is not a square number (or even that it is a prime number).
As  we want to show that  $C_0(X)$ does not detect ideals  in  $F_{\red}^P (\alpha|_{H},u|_{H})$, 
by Lemma \ref{lem:restriction_of_action_to_set} we may assume that $U=X$. Assume also that $P$ is non-empty. Then by Lemma \ref{lem:trivial_action_bundle_description}
we may identify $F_{\red}^P (\alpha|_{H},u|_{H})$ with $C_0\left(\{ F_{\red}^P(H,u_x)\}_{x\in X}\right)$. Pick any $x_0\in X$.
Since $\dim(F_{\red}^P(H,u_{x_0}))=|H|$ is square free, $F_{\red}^P(H,u_{x_0})$ is not simple. Pick any non-trivial ideal $I_{x_0}$ in $F_{\red}^P(H,u_{x_0})$. Then 
$$
I:=\{b\in  C_0\left(\{ F_{\red}^P(H,u_x)\}_{x\in X}\right): b(x_0)\in I_{x_0} \}
$$
is a non-trivial ideal in  $C_0\left(\{ F_{\red}^P(H,u_x)\}_{x\in X}\right)$. It  has zero intersection with $C_0(X)$, which we identify with 
sections $a\in C_0\left(\{ F_{\red}^P(H,u_x)\}_{x\in X}\right)$ such that $a(x)\in \C \cdot 1_{x}$ where $1_{x}$ is the identity in $F_{\red}^P(H,u_{x_0})$.
Hence $C_0(X)$ does not detect ideals in $F_{\red}^P (\alpha|_{H},u|_{H})$. This also gives the claim about $F^{\emptyset}(\alpha|_{H},u|_{H})=F(\alpha|_{H},u|_{H})$  by 
the last part of Remark \ref{rem:Banach_algebra_with_Fourier_decomposition} because $H$ is finite.
\end{proof}

\section{Simplicity and ideal structure}
\label{sec:consequences}
The previous results give immediately the following simplicity criteria. Recall that $\varphi$ is an action of a discrete group $G$ on a locally compact Hausdorff space $X$, and $\alpha$ is the correspodning action on the commutative $C^*$-algebra $C_0(X)$.
\begin{defn}
We say $\varphi$ is \emph{minimal} if there are no non-trivial open  $\varphi$-invariant sets in $X$.
In general, we denote by $\I^\varphi(X)$ the \emph{lattice of all open  $\varphi$-invariant sets} in $X$.
\end{defn}
\begin{lem}\label{lem:simplicity_vs_minimality} Let $F_{\RR}(\alpha,u)$ be a twisted  crossed product admitting the Fourier decomposition
$\{E_t^{\RR}\}_{t\in G}$. 
All  non-trivial ideals in $F_{\RR}(\alpha,u)$ are contained in  $\bigcap_{t\in G}\ker(E_t^{\RR})$ if and only if $\varphi$ is minimal and $C_0(X)\subseteq F_{\RR}(\alpha,u)$ has the generalised intersection property with $\bigcap_{t\in G}\ker(E_t^{\RR})$ as the hidden ideal.
In particular, $F_{\RR}(\alpha,u)$ is simple if and only if  it is reduced, $C_0(X)$ detects ideals in $F_{\RR}(\alpha,u)$ and $\varphi$ is minimal.
\end{lem}
\begin{proof} If all  non-trivial ideals in $F_{\RR}(\alpha,u)$ are contained in  $\bigcap_{t\in G}\ker(E_t^{\RR})$,
then $\bigcap_{t\in G}\ker(E_t^{\RR})$ is the hidden ideal for $C_0(X)\subseteq F_{\RR}(\alpha,u)$, and  $\varphi$ is minimal by Lemma \ref{lem:restriction_of_action_to_set}.
Conversely, assume $\bigcap_{t\in G}\ker(E_t^{\RR})$ is the hidden ideal for $C_0(X)\subseteq F_{\RR}(\alpha,u)$ and let  $J$ be an ideal in $F_{\RR}(\alpha,u)$ which is not contained in $\bigcap_{t\in G}\ker(E_t^{\RR})$. Then $J$ intersects $C_0(X)$ non-trivially and  hence, by Lemma \ref{lem:restriction_of_action_to_set}, $J$ contains $F_{\RR^U}(\alpha^U,u^U)$ for some non-empty $\varphi$-invariant open $U\subseteq X$. Thus  if $\varphi$ is minimal, we necessarily have $J=F_{\RR}(\alpha,u)$, that is $J$  is trivial. 
\end{proof}
\begin{thm}[Simplicity]\label{cor:simplicity_minimality}
Let $\alpha$ be an  action of a discrete group $G$ on $C_0(X)$.
The following conditions are equivalent:
\begin{enumerate}
\item\label{enu:simplicity_minimality1} The dual map $\varphi$ is topologically free and minimal.
\item\label{enu:simplicity_minimality2} All reduced  twisted  crossed products $F_{\RR}(\alpha,u)$, $u\in Z^2(G, C_u(X))$, are  simple.
\item\label{enu:simplicity_minimality3}   $F^P(\alpha)$ is simple for some  (and hence all)  $P\subseteq \{1,\infty\}$.
\item\label{enu:simplicity_minimality4} For any  twisted  crossed product $F_{\RR}(\alpha,u)$, $u\in Z^2(G, C_u(X))$, admitting the Fourier decomposition
$\{E_t^{\RR}\}_{t\in G}$
all  non-trivial ideals in $F_{\RR}(\alpha,u)$ are contained in $\bigcap_{t\in G}\ker(E_t^{\RR})$.
\item\label{enu:simplicity_minimality5}   All  non-trivial ideals in $F^P(\alpha)$  are contained in $\ker(\Lambda_P)$ for some (and hence  all) $P\subseteq  [1,\infty]$.
\end{enumerate}
In particular, for any $P\subseteq [1,\infty]$, the full crossed product $F^P(\alpha)$ is simple if and only if the regular representation $\Lambda_P$ is injective
and $\varphi$ is topologically free and minimal.
\end{thm}
\begin{proof} \ref{enu:simplicity_minimality1} implies \ref{enu:simplicity_minimality2} and \ref{enu:simplicity_minimality4}  by Lemma \ref{lem:simplicity_vs_minimality} and
Corollary  \ref{cor:general_generalised_intersection}. Implications  \ref{enu:simplicity_minimality2}$\Rightarrow$\ref{enu:simplicity_minimality3}
and \ref{enu:simplicity_minimality4}$\Rightarrow$\ref{enu:simplicity_minimality5} are immediate, 
see Remark \ref{rem:properties_of_F^P} and Corollary  \ref{cor:properties_of_F^P}. 
We have \ref{enu:simplicity_minimality3}$\Rightarrow$\ref{enu:simplicity_minimality1}
and \ref{enu:simplicity_minimality5}$\Rightarrow$\ref{enu:simplicity_minimality1} by  Lemma \ref{lem:simplicity_vs_minimality}
and Theorem \ref{thm:generalised_intersection_property}. 
This proves the equivalence of conditions \ref{enu:simplicity_minimality1}-\ref{enu:simplicity_minimality5}. 
The final claim now follows from the equivalence \ref{enu:simplicity_minimality1}$\Leftrightarrow$\ref{enu:simplicity_minimality5}.
\end{proof}

In order to apply our results to non-minimal actions, let us first consider a general  inclusion $A\subseteq B$  of Banach algebras.
Let $\I(A)$  be the set of all (closed, two-sided) ideals in $A$, equipped with
the partial order given by inclusion. This makes $\I(A)$ a complete lattice: for a family of ideals
$\{I_{s}\}_{s\in S} \subseteq \I(A)$, their join in $\I(A)$ 
is the closed linear span $\clsp\{a: a\in I_s, s\in S\}$, and
their meet
 in $\I(A)$ is the intersection
$ \bigcap_{s\in S} I_s$. 
\begin{defn}
Let $A$ be a Banach subalgebra of a Banach algebra $B$. We define maps 
$\Res:\I(B)\to \I(A)$ and $\Ex:\I(A)\to \I(B)$ where, for  $J\in \I(B)$, $I\in \I(A)$, 
$$
\Res(J):=J\cap A\quad\text{ and }\quad \text{$\Ex(I)
$ is the ideal in $B$ generated by $I$}
$$
(if $B$ has an approximate unit, then $\Ex(I)=\overline{BIB}$). We call elements of $\Res(\I(B))\subseteq \I(A)$  \emph{restricted ideals} and elements of $\Ex(\I(A))\subseteq \I(B)$ \emph{extended ideals} for the inclusion $A\subseteq B$.
\end{defn}
Clearly, the maps $\Res$ and $\Ex$ preserve inclusions and for any $I\in \I(A)$, $J\in \I(B)$ we have
$$
I\subseteq \Res(J)\,\, \Longleftrightarrow\,\,  \Ex(I)\subseteq J.
$$
Thus $\Res$ and $\Ex$ form a monotone Galois connection (they are are contravariantly adjoint maps) between the lattices $\I(A)$ and $\I(B)$.  For $C^*$-algebras this was  noticed already by Green, see
\cite[Proposition~9.(i)]{Green:Local_twisted}, and has a number of useful consequences. For instance, $\Ex$ and $\Res$ give an isomorphism 
$\Res(\I(B))\cong \Ex(\I(A))$ and these sets with an inclusion as a preorder are  complete lattices, cf. \cite[Proposition 2.9]{Kwa-Meyer0}.
\begin{defn}[{cf. \cite[Definition 2.11]{Kwa-Meyer0}}]
  \label{def:seperates}
 We say that \emph{$A$ separates ideals in $B$}   if for every two different  $J_1, J_2\in \I(B)$ we have  $J_1\cap A \neq J_2\cap A$.
		\end{defn}
\begin{lem}[{cf. \cite[Proposition 2.12]{Kwa-Meyer0}}]\label{lem:separation_is_residual_detection}
Let $A\subseteq B$ be an inclusion of Banach algebras. The following conditions are equivalent:
\begin{enumerate}
\item\label{enu:separation_is_residual_detection1} $A$ seperates ideals in $B$ (the map $\Res$ is injective).
\item\label{enu:separation_is_residual_detection2} Every ideal in $B$ is extended from $A$ (the map $\Ex$ is surjective).
\item\label{enu:separation_is_residual_detection3} $\I(B)\cong \Res(\I(B))\subseteq \I(A)$ with the isomorphism given by $\Res$.
\end{enumerate}
If $A$ is a $C^*$-algebra, then for every $I\in \Res(\I(B))$, the image of  $A$ in the quotient $B/\Ex(I)$ is naturally isometrically isomorphic to $A/I$ and
each of the above conditions is equivalent to 
\begin{enumerate}\setcounter{enumi}{3}
\item\label{enu:separation_is_residual_detection4}  $A$ `residually detects ideals' in $B$, that is, $A/I$ detects ideals in the quotient $B/\Ex(I)$
 for every $I\in \Res(\I(B))$.
\end{enumerate}
\end{lem}
\begin{proof} Equivalences between \ref{enu:separation_is_residual_detection1}-\ref{enu:separation_is_residual_detection3}  hold for any Galois connection (in particular $\Res$ and $\Ex$ are generalised inverses of each other). Assume now that $A$ is a $C^*$-algebra. Let $I\in \Res(\I(B))$. 
Then $A\cap \Ex(I)=\Res(\Ex(I))=I$ and hence the map $A/I\ni a +I \mapsto a+\Ex(I)\in B/\Ex(I)$ is an injective representation. By minimality of the $C^*$-norm, see Lemma \ref{lem:incompressibility}, this representation is in fact isometric. 
Now assume  that  \ref{enu:separation_is_residual_detection2} fails. Then there is $J\in \I(B)$ with $J\neq \Ex(I)$ where $I:=J\cap A\in \Res(\I(B))$. 
Since $\Ex(I)\subsetneq J$,  the image of $J$ in $B/\Ex(I)$ is a non-zero  ideal that can be identified with the quotient  $J/\Ex(I)$ (and so in particular this image is closed), and whose intersection with $A/I$ is zero.
 Thus $A/I$ does not detect ideals in  $B/\Ex(I)$,  and hence \ref{enu:separation_is_residual_detection4} fails. Conversely, assume that $A/I$ does not detect ideals in  $B/\Ex(I)$ for some $I\in \Res(\I(B))$, so 
there is a non-zero ideal $K\subseteq B/\Ex(I)$ whose intersection with $A/I$ is zero. The preimage of $K$ under the quotient map is an ideal $J\in \I(B)$ such 
that $J\cap A=I$, but $\Ex(I)\subsetneq J$, and hence  \ref{enu:separation_is_residual_detection2} fails.
\end{proof}
Let us now consider the case when  $A=C_0(X)$ and $B=F_{\RR}(\alpha,u)$ is  a Banach algebra crossed product.
It follows from Lemma \ref{lem:restriction_of_action_to_set} that 
$\Res(\I(B))$ can be identified with the lattice $\I^\varphi(X)$ of all open $\varphi$-invariant sets in $X$. It also implies that  
$J\in\Ex(\I(A))$ if and only if $J=F_{\RR^U}(\alpha^U,u^U) $
for some $U\in \I^\varphi(X)$. Thus, in view of Lemma \ref{lem:separation_is_residual_detection} (see statement \ref{enu:separation_is_residual_detection4}),
for separation of ideals in $F_{\RR}(\alpha,u)$ we need the following property, which in the $C^*$-algebraic considerations  is called \emph{exactness}, 
see \cite[Definition 1.5]{Sierakowski} (as it can be used to characterise exactness of the group, see \cite[Theorem 5.1.10]{Brown-Ozawa}). We find the name \emph{residually reduced} more appropriate in our more general context.

\begin{defn}\label{def:residually_reduced} A \emph{crossed product $F_{\RR}(\alpha,u)$ is residually reduced} if 
for every $U\in \I^\varphi(X)$
the quotient Banach algebra $F_{\RR}(\alpha,u)/F_{\RR^U}(\alpha^U,u^U)$ is naturally a reduced crossed product for $(\alpha^Y,u^Y)$ where $Y=X\setminus U$. 
\end{defn}
\begin{rem}
If $F_{\RR}(\alpha,u)$ is residually reduced, then it is reduced   (consider $U=\emptyset$).  When $\varphi$ is minimal
the converse implication  holds trivially (as then $\I^\varphi(X)=\{\emptyset, X\}$). 
\end{rem}
\begin{lem}\label{lem:residually_reduced_characterisation}
$F_{\RR}(\alpha,u)$ is a residually reduced crossed product if and only if $F_{\RR}(\alpha,u)$ admits the Fourier decomposition $\{E_t^{\RR}\}_{t\in G}$  and for any open $\varphi$-invariant subset $U\subseteq X$ we have 
$$
F_{\RR^U}(\alpha^U,u^U)=\{b\in F_{\RR}(\alpha,u):  E_t^{\RR}(b)\in C_0(U) \text{ for all }t\in G\}.
$$
  \end{lem}
\begin{proof} If $F_{\RR}(\alpha,u)$ is a reduced crossed prodcut  it admits the FD. Assume then that $F_{\RR}(\alpha,u)$  admits the FD
and let $U\in \I^\varphi(X)$.
The maps $\{E_t^{\RR}\}_{t\in G}$
 descend to contractive linear maps $E_t^{\RR^Y}:F_{\RR}(\alpha,u)/F_{\RR^U}(\alpha^U,u^U)\to C_0(X)/C_0(U)\cong C_0(Y)$; 
$$
E_t^{\RR^Y}(b+F_{\RR^U}(\alpha^U,u^U)):= E_t^{\RR}(b)+ C_0(U), \qquad b\in F_{\RR}(\alpha,u).
$$
These maps form the FD for $F_{\RR}(\alpha,u)/F_{\RR^U}(\alpha^U,u^U)$ viewed   as the completion of $ F(\alpha^Y,u^Y)$. This quotient crossed product is reduced if and only if $\bigcap_{t\in G} E_t^{\RR^Y}=\{0\}$, which holds  if and only if $F_{\RR^U}(\alpha^U,u^U)=\{b\in F_{\RR}(\alpha,u):  E_t^{\RR}(b)\in C_0(U) \text{ for all }t\in G\}$.
\end{proof}
\begin{prop}\label{prop:residually_reduced_crossed_products}
For any $P\subseteq \{1,\infty\}$ the crossed product $F^P(\alpha,u)$ is always residually reduced. More specifically, for  $U\in \I^\varphi(X)$ and  $Y:=X\setminus U$ the restriction map 
$
\ell^{1}(G,C_0(X))\ni b\mapsto b|_{Y}\in \ell^{1}(G,C_0(Y)) \text{  where  } b|_{Y}(t):=b(t)|_Y,\,\, t\in G
,
$
 yields an
isometric isomorphism 
\begin{equation}\label{eq:isomorphisms_of_quotients}
F^P(\alpha,u)/F^P(\alpha^U,u^U)\cong F^P(\alpha^Y,u^Y).
\end{equation}
\end{prop}
\begin{proof}
Let $b
\in C_c(G, C_0(X))\subseteq F(\alpha,u)$. For any $c\in F(\alpha^U,u^U)$ we clearly have inequalities 
$$
\|b+c\|_{ F(\alpha,u)}=\sum_{t\in G}\|b(t)+c(t)\|_{\infty} \geq  \sum_{t\in G}\|b(t)|_{Y}\|_{\infty} =\|b|_{Y}\|_{F(\alpha^Y,u^Y)},
$$
$$
 \|b+c\|_{F^1(\alpha,u)} =\max_{x\in X} \sum_{t\in G}|[b(t)+c(t)](x)| \geq   \max_{x\in Y} \sum_{t\in G}|b(t)(x)| =\|b|_{Y}\|_{F^1(\alpha^Y,u^Y)},
$$
$$
 \|b+c\|_{F^{\infty}(\alpha,u)} =\max_{x\in X} \sum_{t\in G}|[b(t)+c(t)](\varphi_t(x))| \geq   \max_{x\in Y} \sum_{t\in G}|b(t)(\varphi_t(x))| =\|b|_{Y}\|_{F^\infty(\alpha^Y,u^Y)}.
$$
The set $G_b=\{t\in G: b(t)\neq 0\}$ is finite. For any $\varepsilon>0$ and any $t\in G_b$ there is $c(t)\in C_0(U)$ such that
 $\max_{x\in U}|b(t)(x)+c(t)(x)|\leq \max_{x\in \partial U}|b(t)(x)|+\varepsilon/|G_b| $ where $\partial U:=\overline{U}\cap Y$ is the common boundary of $U$ and $Y$.  
This gives $c\in C_c(G, C_0(X)) \subseteq F(\alpha^U,u^U)$ such that  
$$
\|b+c\|_{ F(\alpha,u)} \leq \|b|_{Y}\|_{F(\alpha^Y,u^Y)}+\varepsilon,\quad  \|b+c\|_{F^1(\alpha,u)} \leq  \|b|_{Y}\|_{F^1(\alpha^Y,u^Y)}  +\varepsilon
$$
and  $\|b+c\|_{F^\infty(\alpha,u)} \leq  \|b|_{Y}\|_{F^\infty(\alpha^Y,u^Y)}  +\varepsilon$. This shows that 
$\inf_{c\in F(\alpha^U,u^U)} \|b+c\|_{ F(\alpha,u)} = \|b|_{Y}\|_{F(\alpha^Y,u^Y)}$
and if $P\subseteq \{1,\infty\}$ is non-empty, then 
$$
\inf_{c\in F^P(\alpha^U,u^U)} \max_{p\in P}\|b+c\|_{ F^p(\alpha,u)}=\max_{p\in P}\|b|_{Y}\|_{F^p(\alpha^Y,u^Y)}.
$$
Hence  the map $b\mapsto b|_{Y}$ induces the surjective isometry  \eqref{eq:isomorphisms_of_quotients}, which clearly is linear and multiplicative, 
and hence it is an isometric isomorphism.
\end{proof}

\begin{lem}\label{lem:separation_is_residual_detection_actions}
Let $F_{\RR}(\alpha,u)$  be a crossed product. The following conditions are equivalent:
\begin{enumerate}
\item\label{enu:separation_is_residual_detection_actions1} $C_0(X)$ separates ideals in $F_{\RR}(\alpha,u)$. 
\item\label{enu:separation_is_residual_detection_actions2}  Every ideal in  $F_{\RR}(\alpha,u)$ is generated by its intersection with $C_0(X)$.
\item\label{enu:separation_is_residual_detection_actions3} We have a lattice isomorphism $\I^{\varphi}(X)\cong\I(F_{\RR}(\alpha,u)) $ given 
by $U\mapsto F_{\RR^U}(\alpha^U,u^U)$.
\item\label{enu:separation_is_residual_detection_actions4}  $C_0(X\setminus U)$ detects ideals in $F_{\RR}(\alpha,u)/F_{\RR^U}(\alpha^U,u^U)$ for every $U\in \I^\varphi(X)$.
\end{enumerate}
If these equivalent conditions hold, then $F_{\RR}(\alpha,u)$ is necessarily residually reduced.
\end{lem}
\begin{proof} By Lemma \ref{lem:restriction_of_action_to_set}, conditions \ref{enu:separation_is_residual_detection_actions1}-\ref{enu:separation_is_residual_detection_actions4} are special cases of conditions  \ref{enu:separation_is_residual_detection1}-\ref{enu:separation_is_residual_detection4} in Lemma \ref{lem:separation_is_residual_detection}.
Assume that $F_{\RR}(\alpha,u)$ is not residually reduced. Then there is $U\in \I^\varphi(X)$ such that $F_{\RR}(\alpha,u)/F_{\RR^U}(\alpha^U,u^U)$ (which is a crossed product for $(\alpha^{X\setminus U},u^{X\setminus U})$ that admits the FD, cf. the proof of Lemma \ref{lem:residually_reduced_characterisation}) is not reduced. Then  $C_0(X\setminus U)$ does not detect ideals in $F_{\RR}(\alpha,u)/F_{\RR^U}(\alpha^U,u^U)$ by Remark \ref{rem:reduced_is_necessary}. This contradicts \ref{enu:separation_is_residual_detection_actions4}.
\end{proof}
\begin{defn}[{\cite[Remark 1.18]{Sierakowski}}]
We say $\varphi$ is \emph{residually topologically free} if for any closed $\varphi$-invariant subset $Y\subseteq X$ the restriction $\varphi^Y:G\to \Homeo(Y)$, $\varphi^Y_t:=\varphi_{t}|_{Y}$,   is topologically free. 
\end{defn}

\begin{thm}[Separation of ideals]\label{thm:ideal structure}
Let $\alpha$ be an  action of a discrete group $G$ on $C_0(X)$.
The following conditions are equivalent:
\begin{enumerate}
\item\label{enu:ideal structure1} The dual action $\varphi$ is residually topologically free.
\item\label{enu:ideal structure2} $C_0(X)$ separates ideals in every  residually reduced  twisted  crossed product $F_{\RR}(\alpha,u)$, so
the map $U\mapsto F_{\RR^U}(\alpha^U,u^U)$ is a lattice isomorphism $\I^{\varphi}(X)\cong\I(F_{\RR}(\alpha,u)) $.
\item\label{enu:ideal structure3} For some (and hence all) $P\subseteq \{1,\infty\}$ the map $U\mapsto F^P(\alpha^U)$ gives $\I^{\varphi}(X)\cong\I(F^P(\alpha))$. 
\item\label{enu:ideal structure5} For any  twisted  crossed product $F_{\RR}(\alpha,u)$,
$C_0(X)$ separates ideals in $F_{\RR}(\alpha,u)$ if and only if $F_{\RR}(\alpha,u)$ is residually reduced.
\end{enumerate}
\end{thm}
\begin{proof}\ref{enu:ideal structure1} implies \ref{enu:ideal structure2} by Lemma \ref{lem:separation_is_residual_detection_actions} and Corollary \ref{cor:general_generalised_intersection}.
Conditions \ref{enu:ideal structure2} and \ref{enu:ideal structure5} are equivalent by the last part of Lemma \ref{lem:separation_is_residual_detection_actions}.
By Proposition \ref{prop:residually_reduced_crossed_products}, \ref{enu:ideal structure2} implies \ref{enu:ideal structure3}.
This proposition combined with Theorem \ref{thm:generalised_intersection_property} shows that  \ref{enu:ideal structure3}  is equivalent to \ref{enu:ideal structure1}.
\end{proof} 
\begin{rem} If the action $\varphi$ is amenable,  its restriction $\varphi^Y$ to any closed $\varphi$-invariant set $Y\subseteq X$ is also amenable. It is a well known $C^*$-algebraic fact that  for every $U\in \I^\varphi(X)$ we have a natural isometric $*$-isomorphism
 $F^2(\alpha,u)/F^2(\alpha^U,u^U)\cong F^2(\alpha^{Y},u^{Y})$ where $Y:=X\setminus U$. Hence if $\varphi$ is amenable, then $F^2(\alpha,u)$ is residually reduced  by Theorem \ref{thm:ambenability_reduced_equals_full}. For $p\neq 1,2,\infty$ this is not clear, and the problem is that we do not know whether the quotient $F^p(\alpha,u)/F^p(\alpha^U,u^U)$ is an $L^p$-operator algebra.
If $G$ is finite, then all Banach algebra crossed products $F_{\RR}(\alpha,u)$ are residually reduced, as they are all equal as topological algebras, see  Remark 
\ref{rem:Banach_algebra_with_Fourier_decomposition}.
\end{rem}
\begin{defn}
Let $A$ be a Banach algebra.  The \emph{prime ideal space} of $A$ is the set $\Prime(A)$ of prime elements in the lattice $\I(A)$ of all closed ideals in $A$.
\end{defn}
Thus $p\in \Prime(A)$ if and only if $p\in \I(A)\setminus \{A\}$  and for any $I_1,I_2\in \I(A)$, the inclusion $I_1\cap I_2\subseteq p$ implies 
that either $I_1\subseteq p$ or $I_2\subseteq p$. We equip $\Prime(A)$ with the topology given by the sets
$$
U_I:=\{p\in \Prime(A): I\not\subseteq p\}, \qquad I\in \I(A).
$$
Note that the map $\I(A)\ni I\to U_I\subseteq \Prime(A)$ maps finite meets into intersections and arbitrary joins to unions. Thus $\{U_{I}\}_{I\in \I(A)}$ is indeed 
a topology on $\Prime(A)$.
\begin{lem}\label{lem:quasi_orbit_map_for_inclusions}
Assume that $A$ is a $C^*$-algebra whose primitive ideal space $\check{A}$ is second countable. Then $\check{A}=\Prime(A)$.
Let $A\subseteq B$ be an inclusion of Banach algebras which is symmetric in the sense that $\Ex(I)=\overline{IB}$ for every $I\in \Res(\I(B))$, cf. \cite[Definition 5.2]{Kwa-Meyer0}.
 \begin{enumerate}
\item\label{enu:quasi_orbit_map_for_inclusions1} For each $p\in \Prime(A)$ there is a  largest restricted ideal $\pi(p)$ in $A$ that is contained in $p$.
Morever, $\pi(p)=\pi(q)$ defines an open equivalence relation $\sim$ on $\Prime(A)$.
\item\label{enu:quasi_orbit_map_for_inclusions2} The map $\varrho:\Prime(B)\ni P\longmapsto [P\cap A]\in \Prime(A)/\sim$ is  well defined and continuous.
\end{enumerate}
\end{lem}
\begin{proof} It is well known that $\check{A}=\Prime(A)$, cf. \cite[Corollary 3.9]{Kwa-Meyer0}.
The rest follows from the results of \cite{Kwa-Meyer0} where both $A$ and $B$ are assumed to be $C^*$-algebras, but a close inspection of the proofs shows that they work when $B$ is an arbitrary Banach algebra. 
Namely, the proof of \cite[Lemma 5.1]{Kwa-Meyer0} (which only uses the properties of Galois connection and that ideals in $A$ have approximate units) shows that for symmetric inclusions, we have
$I\in \Res(\I(B))$ if and only if $\Ex(I)=\overline{IB}$. 
This characterisation of elements in $\Res(\I(B))$ is used in the proof of \cite[Lemma 5.4]{Kwa-Meyer0}, to conclude in \cite[Corollary 5.5]{Kwa-Meyer0}
that the conditions (JR), (FR1) and (MIf) defined in \cite{Kwa-Meyer0} hold.
By \cite[Lemma~4.1 and Theorem 4.5]{Kwa-Meyer0} we have \ref{enu:quasi_orbit_map_for_inclusions1}, 
 $\pi$ induces a homeomorphism from quotient space $\Prime(A)/\sim$ onto the space of prime elements in  $\Res(\I(B))$, and $\Res(\I(B))\cong 
\Ex(\I(A))$ is a spatial locale  (as it is embedded by a frame morphism in the spatial locale $\I(A)$).
Now \ref{enu:quasi_orbit_map_for_inclusions2} follows from  \cite[Lemma~4.8]{Kwa-Meyer0} 
whose proof  uses that $B$ is a $C^*$-algebra only to conclude  that $\Res(\I(B))$ is a spatial locale,
 but as noted above we already have this property in our setting.
\end{proof}
\begin{defn}[cf. {\cite[Definitions 4.4, 4.10]{Kwa-Meyer0}}]
We call 
~\(\Prime(A)/{\sim}\) and  
  \(\varrho\colon\Prime(B) \to \Prime(A)/{\sim}\) in
  Lemma~\ref{lem:quasi_orbit_map_for_inclusions}  the \emph{quasi-orbit
    space} and  the
  \emph{quasi-orbit map} for the inclusion \(A\subseteq B\).
For an action $\varphi:G\to \Homeo(X)$ the \emph{quasi-orbit space} is defined  as 
the quotient space $\OO(X)=X/\sim$ for the equivalence relation
$x\sim y$ if and only if  $\overline{Gx}=\overline{Gy}$, where  $\overline{Gx}:=\overline{\{\varphi_t(x):t\in G\}}$ is the closure of the orbit of $x$, see \cite{Effros_Hahn}.
\end{defn}
\begin{prop}\label{prop:quasi-orbit_map}
Assume $X$ is second countable and  let  $F_{\RR}(\alpha,u)$ be a crossed product. The quasi-orbit space for the inclusion $C_0(X)\subseteq F_{\RR}(\alpha,u)$ coincides with the quasi-orbit space $\OO(X)$ for the action. Thus the  quotient map $X\ni x\mapsto [x]\in \OO(X)$ is open and
we have a continuous quasi-orbit map  
$$
\varrho:\Prime (F_{\RR}(\alpha,u))\ni P\longmapsto [x]\in \OO(X) \qquad \text{where}
\qquad P\cap C_0(X)=C_0(X\setminus \overline{Gx}).
$$ 
Moreover, $\varrho$ is a homeomorphism if and only if $C_0(X)$ separates ideals in $F_{\RR}(\alpha,u)$.
\end{prop}
\begin{proof}
Let $A:=C_0(X)$ and $B:=F_{\RR}(\alpha,u)$. By Lemma \ref{lem:restriction_of_action_to_set}, we have $\Res(\I(B))=\{C_0(U):U\in \I^\varphi(X)\}$ and
$\Ex(C_0(U))=F_{\RR^U}(\alpha^U,u^U)$
for  $U\in \I^\varphi(X)$. For each $U\in \I^\varphi(X)$, $C_0(U)$ and $F_{\RR^U}(\alpha^U,u^U)$ have the same approximate unit. Thus the inclusion $A\subseteq B$ is symmetric   and we may apply Lemma \ref{lem:quasi_orbit_map_for_inclusions}. The canonical homeomorphism $X\ni x\mapsto C_0(X\setminus\{x\})\in \Prime(A)$  descends to a homeomorphism between the orbit quotient spaces, as
for each $x\in X$, the largest restricted ideal contained in $C_0(X\setminus\{x\})$ is $\pi(C_0(X\setminus\{x\}))=C_0(X\setminus\overline{Gx})$, and so
$x\sim y$  if and only if $\pi(C_0(X\setminus\{x\}))=\pi(C_0(X\setminus\{y\}))$. The rest follows now from Lemma \ref{lem:quasi_orbit_map_for_inclusions}.
\end{proof}
\begin{cor}\label{cor:effros-hahn}
Assume $X$ is second countable.  An action $\varphi$ is residually topologically free if and only if for every residually reduced  $F_{\RR}(\alpha,u)$ the quasi-orbit map is a homeomorphism $\Prime (F_{\RR}(\alpha,u))\cong  \OO(X)$.
\end{cor}
\begin{proof} Combine Proposition \ref{prop:quasi-orbit_map} and 
Theorem \ref{thm:ideal structure}.
\end{proof}
\begin{rem}
For $C^*$-algebraic crossed products the homeomorphism between the primitive ideal space and quasi-orbit space was already proved in \cite[Corollary 5.16]{Effros_Hahn} and \cite[Proposition 4.17]{Zeller-Meier} under the additional assumptions  that $G$ is  amenable, countable and acts freely on $X$. That freeness can be relaxed to residual topological freeness was noted in \cite[Corollary 4.9]{Renault91}. 
 Also in \cite[Remark 4.10]{Renault91} it was suggested that the amenability (of the action) of $G$ can be relaxed to exactness of the action, which was formalised  
in \cite[Theorem 1.20]{Sierakowski}. 
\end{rem}


\begin{thebibliography}{fulllllll}

\bibitem[AD02]{Anantharaman-Delaroche} 
C. Anantharaman-Delaroche, \emph{Amenability and exactness for dynamical systems and their $C^*$-algebras}, Trans. Amer. Math. Soc. \textbf{354} (2002), 4153--4178.
  
	\bibitem[AL94]{Anton_Lebed}   A. Antonevich, A. Lebedev, Functional differential equations: I. $C^*$-theory,
 Longman Scientific \& Technical,
Pitman Monographs and Surveys in Pure and Applied Mathematics 70,
 1994.

\bibitem[AO22]{Austad_Ortega}
A. Austad, E. Ortega, \emph{Groupoids and Hermitian Banach *-algebras}
Int. J.  Math.  \textbf{33}, No. 14, (2022) 2250090.

\bibitem[AS93]{Arch_Spiel}  R. J. Archbold, J. S. Spielberg,
\emph{Topologically free actions and ideals in discrete $C^*$-dynamical systems,}
Proc. Edinburgh Math. Soc. (2) \textbf{37} (1993), 119--124.

\bibitem[BK21]{bar_kwa}  K. Bardadyn, B. K. Kwa\'{s}niewski,  \emph{Spectrum of weighted isometries: $C^*$-algebras, transfer operators and topological pressure},
Isr. J. Math. \textbf{246}, no. 1 (2021), 149--210.


\bibitem[BKM23]{BKM1} 
K. Bardadyn, B. Kwa\'sniewski, A. Mckee, 
\emph{Banach algebras associated to  twisted \'etale groupoids:
 inverse semigroup disintegration
and 
representations on $L^p$-spaces}, preprint, Arxiv 2303.09997


\bibitem[Boe20]{Boedihardjo} 
M. Boedihardjo, \emph{$C^*$-algebras isomorphically representable on  $\ell^p$}, Anal. PDE \textbf{13} (2020), no. 7, 2173--2181.

\bibitem[Bon54]{Bonsall}
F. F. Bonsall, \emph{A minimal property of the norm in some Banach algebras}, J.~London Math.\  Soc.\  {\textbf{29}}(1954), 156--164.
\bibitem[BO08]{Brown-Ozawa}
N. P. Brown,   N. Ozawa, $C^*$-Algebras and Finite-dimensional Approximations, Graduate Studies in Mathematics, 88 (American Mathematical Society, Providence, RI, 2008)

\bibitem[BS70]{Busby_Smith}
R. C. Busby, H. A. Smith, \emph{Representations of twisted group algebras}, Trans. Amer. Math. Soc. {\textbf{149}}(1970), 503--537.

\bibitem[BEW18]{Buss_Echt_Will}
A. Buss, S. Echterhoff, R. Willett, \emph{Exotic crossed products and the Baum-Connes conjecture},  J. Reine Angew. Math. \textbf{740} (2018), 111--159.

\bibitem[CGT19]{cgt}
Y. Choi, E. Gardella, H. Thiel, \emph{Rigidity results for $L^p$-operator algebras and applications}, Arxiv preprint math.OA/1909.03612, 2019.

\bibitem[Chu21]{Chung}
Y. C. Chung, \emph{Dynamical complexity and $K$-theory of $L^p$ operator crossed products}. J. Topol. Anal. \textbf{13} (2021), no. 3, 809--841.

\bibitem[DDW11]{DDW} S.~Dirksen, M.~de Jeu, and M.~Wortel,
{\emph{Crossed products of Banach algebras. I}},
preprint (arXiv:1104.5151v2)

\bibitem[EH67]{Effros_Hahn}
E.G. Effros, F. Hahn, Locally compact transformation groups and $C^*$-algebras, Mem. Amer.
Math. Soc. no. \textbf{75}, 1967. 92 pp.

\bibitem[Exe17]{Exel_book}
R. Exel, Partial dynamical systems, Fell bundles and applications, Mathematical Surveys and Monographs, AMS (2017), volume 224

\bibitem[ELQ02]{Exel_Laca_Quigg}
R. Exel, M. Laca and John Quigg, \emph{Partial dynamical systems and $C^*$-algebras generated by partial isometri}es, J. Operator Theory  \textbf{47} (2002), 169--186. 

\bibitem[FD88]{FD}
J. M. G. Fell,  R. S. Doran, Representations of *-algebras, locally compact groups, and Banach *-algebraic bundles, Pure and Applied Mathematics, 125, Academic Press, 1988.


\bibitem[Gar21]{Gardella}
E.~Gardella, \emph{A modern look at algebras of operators on {$L^p$}-spaces}.
Expo. Math., \textbf{39} (2021), 420--453.

\bibitem[GL17]{Gardella_Lupini17}
E.~Gardella and M.~Lupini, \emph{Representations of \'{e}tale groupoids on {$L^p$}-spaces},
  Adv. Math., \textbf{318} (2017), 233--278.

\bibitem[GL16]{Gardella_Thiel16}
E. Gardella, H. Thiel, \emph{Quotients of Banach algebras acting on $L^p$-spaces}, Adv. Math. \textbf{296} (2016), 85--92.

\bibitem[GT19]{Gardella_Thiel19}
E. Gardella, H. Thiel, \emph{Representations of $p$-convolution algebras on $L^q$-spaces}, Trans. Amer. Math. Soc 371 (2019), no. 3, 2207--2236.

\bibitem[GT22]{Gardella_Thiel22}
E. Gardella, H. Thiel, \emph{Isomorphisms of algebras of convolution operators}, Ann. Sci. \'Ecole Nor. Sup. (4) \textbf{55} (2022), no. 5, 1433--1471.

\bibitem[Gre78]{Green:Local_twisted}
P. Green, \emph{The local structure of twisted covariance algebras}, Acta Math. \textbf{140} (1978) 191--250. 


\bibitem[HP15]{PH} S.~Hejazian, S.~Pooya 
\emph{Simple reduced $L^p$-operator crossed products with unique trace},
J. Operator Theory \textbf{74} (2015), no. 1, 133--147.


\bibitem[HO23]{Hetland_Ortega}
E.~V. Hetland and E.~Ortega, \emph{Rigidity of twisted groupoid $L^p$-operator algebras}
 J. Funct. Anal, 285:110037, 2023.
  \bibitem[KT90]{Kawa-Tomi}  
S. Kawamura,   J. Tomiyama, \emph{Properties of topological dynamical systems and corresponding $C^*$-algebras}, Tokyo J. Math. \textbf{13} (1990), 251–257.

 \bibitem[KM18]{Kwa-Meyer-1}
B. K. Kwa\'{s}niewski,  R. Meyer, \emph{Aperiodicity, topological freeness and pure outerness:
from group actions to Fell bundles}, Studia Math. \textbf{241} (2018), no. 3, 257--303.

 \bibitem[KM20]{Kwa-Meyer0}  
B. K. Kwa\'{s}niewski,  R. Meyer, \emph{Stone duality and quasi-orbit spaces for generalised $C^*$-
inclusions}, Proc. Lond. Math. Soc. (3) \textbf{121} (2020), no. 4, 788--827.

  \bibitem[KM21]{Kwa-Meyer}  
	 B. K. Kwa\'{s}niewski,  R. Meyer,
\emph{Essential crossed products by inverse semigroup actions: simplicity and pure infiniteness}, Doc. Math. 26 (2021), 271--335.

 \bibitem[KM22]{Kwa-Meyer2}  
	 B. K. Kwa\'{s}niewski,  R. Meyer,
\emph{Aperiodicity: the almost extension property and uniqueness of pseudo-expectations},
IMRN 2022 (2022), no. 18, 14384--14426.


\bibitem[LY17]{LiYu}
B.~Liao and G.~Yu,
{\emph{K-theory of group Banach algebras and Banach property RD}},
preprint (arXiv: 1708.01982v2 [math.FA]).

\bibitem[ODo75]{OD} D. P. O'Donovan, 
\emph{Weighted shifts and covariance algebras}, 
Trans. Amer. Math. Soc. {\bf 208}  (1975), 1--25. 


\bibitem[OP82]{OP}
D. Olesen and G. K. Pedersen,\emph{ Applications of the Connes spectrum to
$C^*$-dynamical systems. III}, J. Funct. Anal. \textbf{45} (1982), 357--390.

\bibitem[PR89]{Packer_Raeburn}
J. A. Packer, I. Raeburn, \emph{Twisted crossed products of $C^*$-algebras}, Math. Proc. Camb. Soc.\ {\textbf{106}}(1989)


\bibitem[Phi13]{Phillips} N.~C.\  Phillips,
\emph{Crossed products of $L^p$ operator algebras and the $K$-theory of Cuntz algebras on $L^p$ spaces},
2013, 54 pages.  	arXiv:1309.6406 

\bibitem[Phi19]{Phillips19} N.~C.\  Phillips,
\emph{Simplicity of reduced group Banach algebras},
2019, 25 pages.  	arXiv:1909.11278

\bibitem[Ren91]{Renault91}
J. Renault, \emph{The ideal structure of groupoid crossed product $C^*$-algebras}, J. Operator
Theory \textbf{25} (1991), no. 1, 3--36.

\bibitem[R\o r21]{Rordam} 
M. R\o rdam, \emph{Irreducible inclusions of simple $C^*$-algebras}, Enseign. Math. \textbf{69} (2023), no. 3/4, pp. 275--314.

\bibitem[Sie10]{Sierakowski}
A. Sierakowski, \emph{The ideal structure of reduced crossed products}, M\"unster J. Math. \textbf{3} (2010), 237--261.


 \bibitem[Tom92]{Tomiyama}  
  J. Tomiyama, The Interplay Between Topological Dynamics and Theory of $C^*$-Algebras, Lecture Notes Ser., vol. 2,
Res. Inst. Math., Seoul, 1992.

 \bibitem[WZ22]{Wang_Zhu}  Z. Wang, S. Zhu, \emph{On the Takai duality for $L^p$ operator crossed products}, preprint, arXiv:2212.00408




\bibitem[ZM68]{Zeller-Meier}
G. Zeller-Meier, \emph{Produits croises d'une $C^*$-algebre par un groupe d'automorphismes}, J. Math. Pures Appl. \textbf{47} (1968), 101--239.

\end{thebibliography}
\end{document}